\documentclass{article}
\usepackage{amssymb,amsmath,amsthm,enumerate}
\makeatletter

\@addtoreset{equation}{section}
\makeatother

\newcommand{\bY}{\mathbb{Y}}

\newcommand{\be}{\begin{equation}}
\newcommand{\ee}{\end{equation}}
\newcommand{\bes}{\begin{equation*}}
\newcommand{\ees}{\end{equation*}}

\renewcommand{\leq}{\leqslant}
\renewcommand{\geq}{\geqslant}

\renewcommand{\div}{\operatorname{div}}
\newcommand{\trace}{\operatorname{tr}}
\newcommand{\as}{\operatorname{as}}
\newcommand{\dist}{\operatorname{dist}}

\providecommand{\R}{\mathbb{R}}

\providecommand{\N}{\mathbb{N}}

\providecommand{\eps}{\varepsilon}
\def\bX{{\bf X}}

\newcommand\cS{\mathcal{S}}

\newcommand\cW{\mathcal{W}}
\newcommand\cX{\mathcal{X}}

\renewcommand{\leq}{\leqslant}
\renewcommand{\geq}{\geqslant}

\renewcommand{\div}{\operatorname{div}}
\newcommand{\curl}{\operatorname{curl}}
\newtheorem{Theorem}{Theorem}
\newtheorem{Definition}{Definition}

\newtheorem{Lemma}{Lemma}
\newtheorem{Remark}{Remark}
\newtheorem{Examples}{Examples}

\author{Franck Sueur
\\ Laboratoire Jacques-Louis Lions
\\ Universit\'e Pierre et Marie Curie - Paris 6
\\ 175 Rue du Chevaleret
\\ 75013 Paris
\\ FRANCE }
\date{\today}

\date{\today}
\title{Smoothness of the trajectories of ideal  fluid  particles
  with Yudovich vorticities in a planar bounded domain}
\begin{document}
\maketitle

\begin{abstract}
We consider the incompressible Euler equations in a (possibly multiply connected) bounded domain $\Omega$ of  $\mathbb{R}^2$, for  flows with bounded vorticity, for which Yudovich  proved  in  \cite{yudo63}  global existence and uniqueness of the solution.
 We prove that if   the boundary  $\partial \Omega$  of  the domain is $C^\infty$ (respectively  Gevrey of order  $M \geq 1$)   then the  trajectories  of the  fluid particles are $C^\infty$ (resp. Gevrey of order  $M+2$).   Our  results also cover the case of  ``slightly unbounded"  vorticities for which Yudovich extended his analysis in  \cite{yudo95}.
  Moreover if in addition the initial vorticity is H\"older continuous  on a part of $\Omega$ then this H\"older regularity  propagates smoothly along the flow lines.
Finally we observe that if the vorticity is constant in a neighborhood of the boundary, the smoothness of the boundary is not necessary for these results to hold.
 \end{abstract}

\section{Introduction}

We consider the initial-boundary-value problem for the $2$-D incompressible Euler equations in a regular (possibly multiply connected) bounded domain $\Omega$:
\be 
\label{2DincEuler}
\left\{
\begin{array}{ll}
\partial_t u + u \cdot \nabla_x u + \nabla_x p = 0, &  \mbox{ in }  (0, +\infty ) \times \Omega,\\
\text{div }  u =0, &   \mbox{ in } [0, +\infty ) \times\Omega,\\
u \cdot \hat{n} = 0,  &  \text{ on }  [0, +\infty )   \times \partial \Omega ,\\
u(0,x)=u_0(x), &  \text{ on } \{t=0\}\times\Omega.
\end{array}
\right.
\ee
Here, $u=(u_1,u_2)$ is the velocity field, $p$ is the pressure and $\hat{n}$ denotes the unit outward normal to the boundary $\partial \Omega $ of $\Omega$.
 A key quantity in the analysis is the vorticity 
$\omega := \text{curl }  u $,
which satisfies the transport equation:
\bes
\partial_t \omega + u \cdot \nabla \omega = 0, \mbox{ in }  (0, +\infty ) \times \Omega,
\ees
so that, at least formally, the integral over $\Omega$ of any function of the vorticity is conserved when time proceeds.

The global existence and uniqueness of classical solutions to  \eqref{2DincEuler}
were obtained by W. Wolibner  \cite{wolibner} and extended to multiply connected domains by Kato  in  \cite{kato}.
This result was extended by Yudovich \cite{yudo63} to flows such that the  initial  vorticity  (and hence the vorticity at any moment $t$) is bounded.
The corresponding velocity field $u$  is  Log-Lipschitz  so that  there exists a unique  flow map  $\Phi$ continuous from  $\R_+  \times \Omega$ to $\Omega$ such that   
  \be 
\label{flow}
  \Phi(t,x)  =  x + \int^{t}_0 u (s, \Phi(s,x) ) ds .
  \ee
Moreover there exists $c >0$ such that for any $t > 0 $, the vector field  $ \Phi(t,\cdot) $ lies in the  H\"older space $C^{0 , \exp(-ct  \| \omega_0  \|_{L^{\infty}  ( \Omega  )})}  ( \Omega  )$, and an example of Bahouri and Chemin \cite{BahouriChemin} shows that this estimate is   optimal.
Here and in the sequel  we denote $C^{\lambda,r}(\Omega)$, for $\lambda$ in $\N$ and $r \in (0,1)$,
 the H\"older space  endowed with the norm:
\begin{align*}
  \| u   \|_{  C^{\lambda,r} ( \Omega ) } := \sup_{ |\alpha| \leqslant \lambda}   \big(  \|   \partial^\alpha u  \|_{L^\infty (  \Omega ) }  
+ \sup_{ x \neq y \in   \Omega } \frac{ |\partial^\alpha u (x) -  \partial^\alpha u (y)| }{  |x - y|^r }  \big) <  + \infty  ,
\end{align*}
and the notation $C^{\lambda,r}_{\text{loc}} (\Omega_0)$ holds for the space of the functions which are in $C^{\lambda,r} (K)$ for any compact subset $K \subset \Omega $.

In this paper we prove the following result concerning the smoothness in time of the flow map.
\begin{Theorem} \label{start2}
Assume that the boundary  $\partial \Omega$ is $C^\infty$ (respectively Gevrey of order   $M \geq 1$). Then  there exists $c >0$  such that for any  divergence free vector field $u_0$ in  $L^{2} (\Omega)$ tangent to the boundary $\partial \Omega$, with $\omega_0 := \curl u_0 \in L^{\infty}  ( \Omega  )$,
 the flow map $ \Phi$ is, for any $r \in (0,1)$, for any $T >0$,   $C^\infty$  (resp. Gevrey of order   $M+2$) from $ \lbrack 0, T  \rbrack $ to $  C^{0 , \tilde{r} } (\Omega)$, with $ \tilde{r} :=  r \exp(-cT  \| \omega_0 \|_{L^{\infty}  (\Omega )}) $.
\end{Theorem}

 Theorem \ref{start2} extends some previous results on the smoothness of the trajectories of the incompressible Euler equations that we now recall. In  \cite{cheinventiones}, \cite{cheminsmoothness} 
Chemin proved some similar statements for classical solutions in the full space. More precisely, he proves 
that  the flow map $ \Phi$ is $ C^{\infty}$ from $ \lbrack 0, T  \rbrack $, for any $T \in (0, T^* )$ with $T^*$ is the lifetime of the  classical solution,
 to the H\"older space\footnote{One has also to require a decreasing condition at infinity to avoid anomalous solutions, for instance imposing that the velocity field $u$ is  in $L^q (\R^3 )$ with $1<q<+\infty$.} $  C^{ 1,r  } (\R^d)$ for $r \in  (0,1)$ and for  both $d=2$ and  $d=3$.
  These results were improved by  Gamblin \cite{gamblin} and  Serfati  \cite{Serfati1}, \cite{Serfati2}, \cite{Serfati3} who prove that the  flow of classical solutions is analytic and that  the  flow of Yudovich's solutions with bounded vorticity is Gevrey $3$, still for fluids filling the whole space.

  Their results were extended to the case of classical solutions in bounded domains,  in both $2$ and $3$ dimensions, in  \cite{katoana}  by Kato (the flow map $ \Phi$ is $ C^{\infty}$ from $ \lbrack 0, T  \rbrack $ to the H\"older space  $  C^{1,r} (\Omega)$ for $r \in  (0,1)$) and in  \cite{ogfstt} (the flow map $ \Phi$ is analytic from $ \lbrack 0, T  \rbrack$ to the H\"older space $  C^{ 1,r  } (\Omega)$ for $r \in  (0,1)$). Actually the main result in  \cite{ogfstt}  is that  the motion of a
rigid body immersed in an incompressible perfect fluid which occupies a three dimensional
bounded domain is  at least as smooth as the boundaries (of the body and of the domain) when  the initial velocity of the fluid is in the  H\"older space  $  C^{1,r} (\Omega)$  (till the classical solution exists and 
till the solid does not hit the boundary).
One ingredient of the proof was precisely the smoothness of the flow  of the incompressible Euler equations. 
We therefore hope that the analysis of the paper should be applied to the smoothness of the motion of a body immersed in a perfect incompressible fluid with Yudovich  vorticities.

The following result bridges Theorem \ref{start2} and the earlier results about classical solutions, proving that, for Yudovich solutions,  extra local H\"older regularity  propagates smoothly along the flow lines.
\begin{Theorem} \label{start2localh}
Under the (respective) hypotheses of Theorem \ref{start2}, and assuming moreover that the restriction  $\omega_0 |_{ \Omega_0 }$ is in the H\"older space $C^{ \lambda_0 ,r}_{loc}  ( \Omega_0 )$, where $ \lambda_0  \in  \N $ and   $r \in (0,1)$ and $\Omega_0$ an open set such that $\overline{\Omega_0} \subset \Omega $, we have that 
 the flow map $ \Phi$ is, for any $T >0$,  for any compact $K  \subset {\Omega_0}$, 
  $C^\infty$  (resp. Gevrey of order   $M + 2 + ( \lambda_0 + 1) ( r + 1)$) from $ \lbrack 0, T  \rbrack $ to $  C^{ \lambda_0 + 1 ,r} (K)$.
\end{Theorem}
Actually we will  obtain Theorem \ref{start2} (resp. Theorem \ref{start2localh}) as a particular case of Theorem \ref{start22} (resp. Theorem \ref{start22localh}) below, which encompasses more general initial vorticities. More precisely we will also consider the  ``slightly unbounded" vorticities introduced by Yudovich  in  \cite{yudo95}.

 \section{Yudovich's slightly unbounded vorticities }

In this section we recall the setting of Yudovich's paper  \cite{yudo95}
with a few extra remarks which will be useful  in the sequel.
 We start with the following definition.
\begin{Definition}[Admissible germs] 
 \label{germi}
 A function  $\theta : [p_0 , +\infty ) \rightarrow  (0, +\infty )$, with  $p_0 > 1$, is said admissible if the auxiliary function $T_{\theta} :  [1, +\infty ) \rightarrow  (0, +\infty )$ defined for $a > 1$ by
\begin{eqnarray*}
T_{\theta} (a) := \inf \{ \frac{a^\epsilon }{\epsilon } \theta (\frac{1 }{\epsilon } ), \ 0 < \epsilon \leqslant 1/ p_0 \} 
\text{ satisfies }
\int^{+ \infty }_1  \frac{da}{a T_{\theta} (a) } = \infty .
\end{eqnarray*}
\end{Definition}
Let us denote $\theta_0 (p) :=  1 $, and, for any $m \in \N^*$,  
\begin{eqnarray}
 \label{thetam}
\theta_m (p) := \log p \cdot \log^2 p \cdot \cdot \cdot  \log^m p ,
\end{eqnarray}
where $\log^m$ is log composed with itself $m$ times.
\begin{Examples}
 \label{exam}
For  any $m \in \N$, the germs  $\theta_m $  are admissible.
\end{Examples}
\begin{proof}
Let be given   $\theta : [p_0 , +\infty ) \rightarrow  (0, +\infty )$, with  $p_0 > 1$.
For $a > e^{p_0}$,
\begin{eqnarray}
 \label{TT}
T_{\theta} (a) \leq e  \log a \cdot   \theta ( \log a ) ,
\end{eqnarray}
since $\epsilon = ( \log a )^{-1}$ is in $(0,1/ p_0 )$.
 Therefore,
\begin{eqnarray*}
\int^{+ \infty }_1  \frac{da}{a T_{\theta} (a) }  \geq e^{-1} \int^{+ \infty }_{ e^{p_0}}  \frac{d \log a}{ \log a \cdot  \theta ( \log a ) }   =     e^{-1} \int^{+ \infty }_{ {p_0}}  \frac{dp}{  p \cdot  \theta ( p ) } .
\end{eqnarray*}
The result  follows from a repeated change of variables: for  any $m \in \N$, for  any $p_2 
 \geq p_1  \geq \exp^{m} (1) $, where $\exp^m$ is $\exp$ composed with itself $m$ times,
\begin{eqnarray}
 \label{calcul}
\int^{p_2 }_{p_1 }   \frac{dp}{  p \cdot  \theta_m ( p ) } =  \log^{m+1} p_2 - \log^{m+1} p_1 .
  \end{eqnarray}
\end{proof}
\begin{Definition}[The space  $\bY_\theta$] 
Given an admissible germ $\theta : [p_0 , +\infty ) \rightarrow  (0, +\infty )$, with  $p_0 > 1$,  we denote $\bY_\theta$ the space of  the divergence free vector fields $u$ in  $L^{2} (\Omega)$ tangent to the boundary, such that  $\curl u $  belongs in
$\cap_{p \geqslant p_0} \, L^p (\Omega)$, and such that there exists $c_f > 0$ such that 
\begin{eqnarray}
 \label{propre}
\| f  \|_{  L^p  ( \Omega  )} \leq c_f \theta (p)  \text{ for } p \geqslant p_0 .
  \end{eqnarray}
It is a Banach space endowed with the following norm:
\begin{eqnarray*}
\| f  \|_{ \bY_\theta  } :=   \| f  \|_{ L^2 ( \Omega) } +     \inf \{ c_f > 0 /  \text{  \eqref{propre} holds true }  \} .
  \end{eqnarray*}
\end{Definition}
\begin{Remark}
 \label{class}
In particular  for $\theta = \theta_0 $, the space $\bY_\theta$ corresponds to  the space of  the divergence free vector fields $u$ in  $L^{2} (\Omega)$ tangent to the boundary with  $\curl u $ in $L^{\infty} (\Omega)$.
\end{Remark}
\begin{Remark}
Vorticities with a point singularity at $x_0 \in \Omega$ of type $\log \log \|  x - x_0 \|^{-1} $ belongs to the space  $\bY_\theta$, with $\theta$ of the form $\theta= c \theta_1 $ (where $c$ is a   positive constant), which is therefore admissible  (cf. \cite{yudo95}, Example $3.3$).
On the other hand thanks to Laplace's method, we have that $L^p$ norms of a vorticity with a point singularity at $x_0 \in \Omega$ of type $\log \|  x - x_0 \|^{-1} $ is equivalent to $c p$ (where $c$ is a   positive constant), and is therefore   non-admissible  (cf. \cite{yudo95}, Example $3.2$).
\end{Remark}
In this setting existence and uniqueness holds according to the following result.
\begin{Theorem}[Yudovich \cite{yudo95}]
 \label{95}
Assume that  the boundary  $\partial \Omega$  of  the domain is $C^2$.
Given $u_0$ in  $\bY_\theta$,  there exists a unique weak solution $u$ of  \eqref{2DincEuler} in  $ L^\infty  ([0, +\infty ), \bY_\theta )$.
\end{Theorem}
We are now going to examine the  flow map of  these solutions.
We first recall the following definition.
\begin{Definition}[Modulus of continuity]
We will say that, $a >0$ being given, a function $\mu : \lbrack 0, a  \rbrack \rightarrow \R_+$ is a modulus of continuity if it is an increasing continuous function such that $\mu (0) = 0$. We will denote by $C_\mu (\Omega)$ the space of 
continuous   functions $f$ over $\Omega$  such that the following semi-norm
\begin{eqnarray*}
\| f  \|_{C_{\mu}  ( \Omega  )} := 
 \sup_{ 0 < \| x - y \| \leqslant a } \frac{ \| f(x) - f(y) \| }{\mu( \| x - y \|) } 
\end{eqnarray*}
is finite.
\end{Definition}
A function in $C_\mu (\Omega)$ extends uniquely to a function in  $C_\mu (\overline{\Omega})$. In particular, since $\Omega $ is bounded,  a  function in $C_\mu (\Omega)$ is bounded.
Moreover  in the case where $a > \text{diam}( \Omega)$, and $\mu (h) := h^r$ with $r \in (0,1)$, then  $C_\mu (\Omega) = C^{0,r} (\Omega)$.
  
Let us remark the following.
\begin{Lemma}
 \label{trivcomp}
Let $F$ be in the  H\"older space $C^{0,r}(\Omega)$,  with $r \in (0,1)$,  let $\mu$ be a modulus of continuity and let $\phi :\Omega \rightarrow \Omega$ be  in   $C_\mu (\Omega)$; then $\mu^r$ is a modulus of continuity and $F \circ \phi  $ is in $C_{\mu^r} (\Omega)$ with $ \|  F \circ \phi  \|_{ C_{\mu^r} (\Omega) } \leqslant   \|  F   \|_{ C^{0,r}(\Omega) } \cdot \| \phi   \|_{C_{\mu}  ( \Omega  )}^r$.
\end{Lemma}
\begin{proof}
For $x,y $ in $\Omega $, with $0 < \| x - y \| \leqslant a$, we have
\begin{eqnarray*}
\frac{\| F \circ \phi (x) - F \circ \phi (y) \|}{\mu( \| x - y \|)^r }
  \leqslant   \frac{\| F \circ \phi (x) - F \circ \phi (y) \| }{  \| \phi (x) -  \phi (y) \|^r } \cdot  \frac{\| \phi (x) -  \phi (y) \|^r }{\mu( \| x - y \|)^r }
  \leqslant  \|  F \|_{ C^{0,r}(\Omega) } \cdot  \| \phi   \|_{C_{\mu}  ( \Omega  )}^r ,
\end{eqnarray*}
and $\mu^r : h \in \lbrack 0, a  \rbrack \mapsto \mu(h)^r $ is a modulus of continuity.
\end{proof}
\begin{Definition}[Osgood modulus of continuity]
 We say that  a  modulus of continuity $\mu : \lbrack 0, a  \rbrack \rightarrow \R_+$ is an Osgood modulus of continuity if $\int_0^{a }  \frac{dh}{ \mu(h) } = + \infty $.
\end{Definition}
\begin{Remark}
For a modulus of continuity  (Osgood or not), only the behavior of $\mu$ near $0$ does really matter for our purposes, not the value of $a$.
\end{Remark}
\begin{Theorem}[Yudovich  \cite{yudo95}]
\label{mod}
Assume that  the boundary  $\partial \Omega$  of  the domain is $C^2$.
There exists $C > 0$ (depending only on $\Omega$) such that 
for any admissible germ $\theta$, for any initial velocity $u_0$ in  $\bY_\theta$,  
 the corresponding  unique weak solution $u$ of the Euler equations provided by Theorem \ref{95} is in   $ L^\infty  ([0, +\infty ), C_\mu ( \Omega  ))$, where the modulus of continuity  $\mu$ satisfies 
 \begin{eqnarray}
 \label{maj}
  \mu (h) \leq  C hT_{\theta} (h^{-2}), 
   \end{eqnarray}
  where $T_{\theta}$ is the function which appeared in Definition  \ref{germi}. Thus $ \mu$ is  an Osgood modulus of continuity.
\end{Theorem}
Let us stress that the function $\mu$ in Theorem \ref{mod} is independent of time.
\begin{Remark}
Theorem \ref{mod} applies in particular for  $\theta =  \theta_0$ (bounded vorticity), and we recover the  well-known fact since  \cite{yudo63} that we can take an Osgood modulus of continuity of the form 
$$ \mu (h):=  C  \| \omega_0  \|_{L^{\infty}  ( \Omega  )} h  \log (h^{-2} ), \text{ with } C >0 \text{  (depending only on } \Omega).$$
\end{Remark}
It is therefore classical (see for example \cite{bcd}) that, for $u$ as in Theorem \ref{mod},   there exists a unique corresponding flow map  $ \Phi$ continuous from  $\R_+  \times \Omega$ to $\Omega$ such that   
 \begin{eqnarray*}
  \Phi(t,x)  =  x + \int^{t}_0 u (s, \Phi(s,x) ) ds .
 \end{eqnarray*}
   This relies on the Osgood Lemma that we recall above under a form appropriated for the sequel. 
\begin{Lemma}
\label{Olemma}
 Let $\rho$ be a measurable function from $ \lbrack 0, T  \rbrack $ into $ \lbrack 0, a  \rbrack $, $\mu$ a modulus of continuous on  $ \lbrack 0, a  \rbrack $, $c \in  ( 0, +\infty)$ and assume that for all $t  \in  \lbrack 0,T  \rbrack $,
 \begin{eqnarray*}
 \rho (t) \leqslant c + \int_0^t  \mu ( \rho (s)) ds .
 \end{eqnarray*}
Then for all $t  \in  \lbrack 0,T  \rbrack $,
 \begin{eqnarray*}
  \int_c^{ \rho (t)}   \frac{dh}{ \mu(h) }  \leqslant t .
 \end{eqnarray*}
 \end{Lemma}
Let $T>0$. Since the  modulus of continuous $ \mu$ provided by Theorem \ref{mod} is an Osgood modulus there exists $ \tilde{a} \in (0,a)$ such that 
 \begin{eqnarray*}
    \int_{\tilde{a}}^{a}   \frac{dh}{ \mu(h) } \geqslant \kappa T , \text{ where }  \kappa :=  \|  u \|_{L^\infty  ([0, +\infty ), C_\mu ( \Omega  ))} .
 \end{eqnarray*}
   For any $t  \in    \lbrack 0,T  \rbrack $, for any $h \in  (0, \tilde{a}  \rbrack$ there exists an unique $\Gamma_t (h)  \in   \lbrack h, a  \rbrack $
   such that  
 \be 
 \label{modulus}
    \int_{h}^{\Gamma_t (h)}   \frac{dh}{ \mu(h) } =  \kappa  t .
 \ee
   In addition, for any $t  \in    \lbrack 0,T  \rbrack $, extended by $\Gamma_t (0) = 0$, the function  $\Gamma_t $ is a modulus of continuity.
   Furthermore, we have the following.
\begin{Lemma}
\label{Gamma}
For  any $t  \in    \lbrack 0,T  \rbrack $, the  flow map $ \Phi(t,.)$ at time $t$ belongs to $C_{\Gamma_t} (\Omega)$.
 \end{Lemma}
   \begin{proof}
   Let $x,y$ be in $\Omega$ with  $0 < \| x - y \|  \leqslant \tilde{a}$. Using Theorem \ref{mod} we get that for any $t  \in    \lbrack 0,T  \rbrack $, 
 \begin{eqnarray*}
 \| \Phi(t,x)  -   \Phi(t,y) \|  \leqslant \| x - y \| +  \kappa  \int_0^t  \mu (  \|   \Phi(s,x) - \Phi(s,y)  \| ) ds .
   \end{eqnarray*}
   Thanks to the  Lemma \ref{Olemma}, we infer that for any $t  \in    \lbrack 0,T  \rbrack $, 
 \begin{eqnarray*} 
    \int_{  \| x - y \|}^{  \|   \Phi(t,x)  -   \Phi(t,y)  \|  }   \frac{dh}{ \mu(h) }  \leqslant  \kappa   t ,  \ 
   \text{ and thus  }
   \|   \Phi(t,x)  -   \Phi(t,y)  \|  \leqslant \Gamma_t (\| x - y \| ).
     \end{eqnarray*}
   \end{proof}
     At this point we can already say that for any $x\in \Omega$, the curve $t \mapsto  \Phi(t,x) $ is absolutely continuous hence differentiable almost everywhere with  $\partial_t  \Phi(t,x)  = u (t, \Phi(t,x) )$.
In addition uniqueness implies that the flow satisfies the Markov semigroup property.
Finally for any $t  \in  \lbrack 0,T \rbrack $, the flow map  $ \Phi(t,\cdot) $ at time $t$ is a volume-preserving homeomorphism.

Let us now have a deeper look at the smoothness in space of the flow map.
We start with recalling the following.
\begin{Remark}
 \label{bahou}
In the particular case where the initial vorticity is bounded  (when $\theta = \theta_0 $) it is well-known that 
 there exists $c >0$ such that for any $t > 0$, $ \Phi(t,\cdot) $ lies in the  H\"older space $C^{0, \exp (-ct  \| \omega_0  \|_{L^{\infty}  ( \Omega  )})}$.
 An example by Bahouri and Chemin  \cite{BahouriChemin} shows that this estimate is actually optimal.
   \end{Remark}
\begin{Lemma}
\label{loglog}
For $\theta = \theta_m $, as in 
 \eqref{thetam}, with   $m \in \N^*$,  the modulus of continuity 
 $\Gamma_t $ of the flow map satisfies
 \begin{eqnarray}
 \label{log2}
  \Gamma_t  (h) \leq (\exp^m (( \log^m (h^{-2}) )^{\exp (-2C \kappa t) } ))^{ -\frac{1}{2} } .
   \end{eqnarray}
 \end{Lemma}
   \begin{proof}
   Combining  \eqref{TT} and  \eqref{maj}, we get 
  $\mu (h) \leq  Ce h \theta_{m+1} (h^{-2})$.
Hence, using  \eqref{modulus}, 
 \bes
\kappa t =    \int_{h}^{\Gamma_t (h)}   \frac{dh}{ \mu(h) } \geq -  \tilde{C}^{-1}  \int_{h^{-2}}^{\Gamma_t (h)^{-2}}   \frac{dp}{  p \cdot  \theta_{m+1} ( p )  }
\\ =  \tilde{C}^{-1} (-  \log^{m+2}  (\Gamma_t (h)^{-2} ) +  \log^{m+2}   (h^{-2} )),
 \ees
by \eqref{calcul}, with $ \tilde{C} := 2Ce$. The result is then straightforward.
 \end{proof}
   \cite{kelli},  Chapter 5  provides some examples of Yudovich's slightly unbounded  initial vorticities for which the solution to the Euler equations in the plane has an associated 
flow which lies in no H\"older  space of positive exponent for all
positive time. However for $\theta = \theta_m $, with   $m \in \N^*$,  the flow map is necessarily Dini continuous. Before to prove this, let us recall the following  definition.
\begin{Definition}[Dini modulus of continuity]
 We say that  a  modulus of continuity $\mu : \lbrack 0, a  \rbrack \rightarrow \R_+$ is a Dini modulus of continuity if 
$\int_0^{a }  \frac{\mu(h)}{ h}  dh < + \infty $.
A function which belongs to a space $C_{\mu}$ where $\mu$  is a Dini modulus of continuity is said Dini continuous.
\end{Definition}
\begin{Lemma}
\label{mdini}
For $\theta = \theta_m $ with  $m \in \N^*$, $r \in (0,1\rbrack$ and $t\geqslant 0$,   the modulus  of continuity
 $\Gamma_t^r $ of the flow map is Dini.
      \end{Lemma}
  \begin{proof}
  Let  $m \in \N^*$ and $t \geqslant 0$. We set $u :=  \log^m (h^{-2}) $ so that,  using \eqref{log2}, we get 
  \begin{eqnarray*}
2  \int_0^{a }  \frac{\Gamma_t^r (h)}{ h}  dh  \leqslant 
   \int_{  \log^m (a^{-2}) }^{+\infty  }  \prod_{i=1}^{m-1} \exp^i (u) \cdot  \exp (-\frac{r}{ 2}    \exp^{m-1} (u^{\exp (-2Ct) }  )) du < +  \infty .
    \end{eqnarray*}
  \end{proof}
%
 We do not know if the function $\Gamma_t$, with $t \geq 0$, defined by  \eqref{modulus}, for $h$ small enough, is necessarily a Dini modulus of continuity when $\mu : \lbrack 0, a  \rbrack \rightarrow \R_+$ is an Osgood modulus of continuity, in other terms is the flow map generated by an Osgood vector field is necessarily Dini continuous.
  
 \section{Statement of the results}
 \label{gene}
 
 Our analysis applies as well to  the case of multiply connected domains.
Let  us therefore assume that $\Omega $ has as internal boundaries some piecewise smooth Jordan curves 
$C_1$, ..., $C_d$, and is bounded externally by a closed curve $C_0$.
We choose the positive directions on the curves $C_0$, $C_1$..., $C_d$ such that the domain  $\Omega $ is always on the left (so that the  curves $C_1$, ..., $C_d$  are oriented clockwise and the curve $C_0$ is oriented counter-clockwise). 
For a smooth enough function $f$, we denote  $ \Gamma_i (f)$ 
 the circulations of $f$ around the curve $C_i$, for $1 \leqslant  i \leqslant  d$.

 \subsection{Smoothness of the trajectories  for bounded or ``slightly unbounded" vorticities }
 \label{gene1}

We are now ready to state the first general result hinted in the introduction.
\begin{Theorem}
 \label{start22}
Let   $\theta$  be an admissible germ and assume that the initial data $u_0$ is in  $\bY_\theta$.
Assume that the boundary  $\partial \Omega$ is $C^\infty$ (respectively Gevrey of order   $M \geq 1$).
Then the flow map $ \Phi$ is, for any $r \in (0,1)$, for any $T >0$,   $C^\infty$  from $ \lbrack 0, T  \rbrack $ to $ C_{  \Gamma_T^r  } (\Omega)$  (resp. satisfies there exists $L >0$ depending only on $\Omega$ such that for any $t \in \lbrack 0, T  \rbrack$, for any $k \in \N$,
 \begin{eqnarray}
 \label{start22beurk}
      \| \partial^{k +1 }_t \Phi (t,.) \|_{ C_{\Gamma_t^r }   (\Omega)}
       \leqslant  \frac{(k!)^{M+1}   {L}^{k+1}}{(1-r)^{k}}    \Big(  \frac{k!}{(1-r)^{k+1}} \theta\big( \frac{2(k+1)}{1-r} \big)^{k+1} +   \sum_{i=1}^d  | \Gamma_i ( u_0)  |^{k+1}  \Big).
\end{eqnarray}
\end{Theorem}

To deduce Theorem \ref{start2} from  Theorem \ref{start22} it suffices to take into account  Remark  \ref{class} and  Remark   \ref{bahou}. 
Let us stress that Theorem \ref{start22} fails to prove that the flow is  $C^\infty$  from $ \lbrack 0, T  \rbrack $ to $ C_{  \Gamma_T  } (\Omega)$, for the estimate  \ref{start22beurk} blows up when $r$ tends to $1$.
In the particular case where  $\theta :=  \theta_m $, with   $m \in \N$, the estimate \eqref{start22beurk} and Lemma \ref{mdini} yield that the flow map is, for any $\eps \in (0,1)$, 
 for any $T >0$,  Gevrey of order   $M + 2 + \eps$  from $ \lbrack 0, T  \rbrack $ to the space $C_D  (\Omega)$ of Dini continuous functions.

 \subsection{Extra local  H\"older regularity  propagates smoothly }
 \label{gene2}

In this section we deal with the case where the initial vorticity is locally  H\"older continuous.
We will prove first the following result:
\begin{Theorem} \label{localreg}
Assume that  the boundary  $\partial \Omega$  of  the domain is $C^2$.
Assume that the initial data $u_0$ is in  $\bY_{\theta_m }$, with   $m \in \N$.
Assume that $\Omega_0$ is a  open subset such that 
 $\overline{ \Omega_0 } \subset  \Omega $. Assume that  $\omega_0 |_{ {  \Omega_0}}$ is in $C^{\lambda_0 ,r}_{\text{loc}} (  \Omega_0 )$ with $\lambda_0$ in $\N$ and $r \in (0,1)$.  Then for any $t \geqslant 0$,  the restrictions $\omega (t,\cdot )|_{\Omega_t  }$ and $u (t,\cdot ) |_{\Omega_t  }$ of the vorticity and of the velocity to the set 
  $\Omega_t := \{  \Phi(t,x) , \/ \ x \in \Omega_0 \} $ are respectively in $C^{\lambda_0 ,r}_{\text{loc}}  ( \Omega_0 ) $ and $C^{\lambda_0 +1,r}_{\text{loc}}  ( \Omega_0 )$.
 \end{Theorem}
Theorem \ref{localreg} is a slight extension of  Proposition $8.3$ of \cite{BM} which deals only with the case $\theta = \theta_0$, that is with bounded vorticities.

Under the assumptions of Theorem \ref{localreg}, if the boundary is smooth,  local H\"older regularity  propagates smoothly along the flow lines.
\begin{Theorem} \label{start22localh}
Under the hypothesis of Theorem \ref{localreg}, assuming moreover that  the boundary  $\partial \Omega$ is $C^\infty$ (respectively Gevrey of order   $M \geq 1$), then 
 the flow map $\Phi$ is, for any $T >0$,  for any compact $K  \subset {\Omega_0}$, 
  $C^\infty$  from $ \lbrack 0, T  \rbrack $ to $  C^{\lambda_0 + 1,r} (K)$ (resp. satisfies there exists $L >0$  such that for any $t \in \lbrack 0, T  \rbrack$, for any $k \in \N$,
 \begin{eqnarray}
 \label{dumasla}
      \| \partial^{k +1 }_t \Phi (t,.) \|_{ C^{\lambda_0 + 1,r} (K)  }       \leqslant  
   L^{k+1} (k!)^{M+1+ (\lambda_0 +1) (r +1)}       \Big(   \| u (t,.) \|_{ C^{ \lambda_0 +1 ,r} (\underline{K}_t )} +   \| u (t,.) \|_{ W^{1, \frac{2(k+1)}{1-r} } ( \Omega  )  }    \Big)^{k+1} ,
\end{eqnarray}
where $\underline{K}_t := \{  \Phi(t,x) , \/ \ x \in \underline{K} \} $ with $\underline{K}$  a compact such that  $K  \subset \dot{\underline{K}}  \subset \underline{K} \subset {\Omega_0}$.
\end{Theorem}
After the proof of Theorem \ref{start22localh}, it would be clear that 
Theorem \ref{start22localh} yields Theorem \ref{start2localh} when we consider the case $m=0$ observing that last factor of the right hand side of  \eqref{dumasla} can be therefore estimated by the initial vorticity with a extra factor $(k!)$.

 \subsection{A few remarks about weaker solutions and the influence of the boundary smoothness }
 \label{gene4}

 Actually what we really need in the proof of Theorem \ref{start22} is, first, of course, the existence of a flow and that the vorticity lies in any $ L^{p}  (\Omega)$ for large $p$.
 However Theorem \ref{start22} does not cover some cases where a flow map can be defined, and even uniquely. In particular in \cite{Vishik} Vishik proves the following result of
 existence and uniqueness of  solutions to the  $2$D incompressible Euler  equations in the full plane in a borderline space of Besov type.
\begin{Theorem}[Vishik] \label{Vish}
Assume that $\Omega := \R^2$ and that $\omega_0 \in L^{p_0} ( \R^2 ) \cap L^{p_1} ( \R^2 )$ with $1 < p_0 <2 < p_1 < + \infty $.
Assume moreover that $\omega_0$ is in 
\be 
\label{BGAM}
\begin{array}{l}
B_\Gamma := \{ f \mbox{ in }  \cS' ( \R^2 ) \mbox{ s.t. } 
\sum_{j=-1}^N \| \Delta_j f \|_{L^{\infty} ( \R^2 )} = O(\Gamma(N)) \} ,
\end{array}
\ee
where $\Gamma (N) := \log N $ and the $ \Delta_j f$ denote the terms in the Littlewood-Paley decomposition of $f$.
Let $u_0$ be the velocity associated to  $\omega_0$ by the Biot-Savart law. 
Then there exists $T>0$ and a solution to the Euler equations  \eqref{2DincEuler} satisfying $$\omega \in L^{\infty} (0,T; L^{p_0} ( \R^2 ) \cap L^{p_1} ( \R^2 )) \cap C_{w*}  ( \lbrack 0, T  \rbrack; B_{\Gamma_1} ),$$
where $\Gamma_1 (N ) := (N+2) \Gamma (N)$.
The corresponding velocity $u$ is in $ L^{\infty} (0,T; C_\mu ( \R^2 ))$ with $\mu (r) := r \cdot \log r^{-1} \cdot \log^{2}  r^{-1} $, so that the flow map is uniquely defined.
\end{Theorem}
It is proved in \cite{Vishik}, Proposition $2.1$ that for any $\rho > 1$, there exists $f \in B_\Gamma $ and in $ \cap_{ 1\leqslant p <  \rho } L^{p} $ but not in $ \cap_{ p \geqslant   \rho } L^{p} $. Therefore our proof of Theorem \ref{start22} based on the scale of the Lebesgue spaces $L^{p} $ is not adapted to tackle Vishik's solutions. 
However the smoothness of the flow map in this case can be deduced from Gamblin's work  \cite{gamblin}. 
\begin{Theorem} \label{VishGam}
Under the assumptions of Theorem  \ref{Vish}, the flow map is for any $\eps \in (0,1)$, 
 Gevrey of order   $3 + \eps$  from $ \lbrack 0, T  \rbrack $ to the space $C_D  (\R^2)$ of Dini continuous functions.
\end{Theorem}
\begin{proof}
According to  \cite{gamblin}, estimate $(2.3)$, there holds for any $k  \in  \N$, for any  $\eps _1 \in (0,1/(2(k+1)))$, for any $t \in  \lbrack 0, T  \rbrack $, 
$$  \| D^k u  \|_{C^{0,1- \eps_1 (k+1) }  (\R^2 ) } \leqslant \| u  \|_{C^{0,1- \eps_1 }  (\R^2 ) }
(C  \eps^{-1}_1  \| u  \|_{C^{0,1- \eps_1 }  (\R^2 ) } )^k \frac{k!}{(k+2)^2} .$$
It then suffices to take $\eps_1 := \eps /  (k+1)$, to use the embedding
$$  \| \cdot  \|_{C^{0,1- \eps_1 }}  \leqslant C \eps^{-1}_1 \log \eps^{-1}_1 \|  \cdot  \|_{ C_\mu  } ,$$
Lemma  \ref{trivcomp} and  \ref{mdini} to conclude.
\end{proof}
Moreover for any initial vorticity in  $ L^{p}  (\Omega)$, with $p >2$, one gets a corresponding velocity which is continuous so that Peano's theorem applies and provides the existence of a flow. Furthermore it is known since a bunch of papers by Kisielewicz in $1975$ that uniqueness is generic in the sense of Baire's category for Peano's continuous vector-fields (see Bernard's paper \cite{bernard} Theorem $1$  for a more procurable proof).
Let us also refer here to the renormalization theory by Di Perna-Lions \cite{DPL} and Ambrosio  \cite{Ambrosio}  for some properties of the flow map up to some zero Lebesgue measure sets.

Next Theorem provides some examples of even weaker solutions than in Theorem \ref{start22}  for 
which some flow lines  are analytic, despite the boundary is only assumed to be $C^2$.
\begin{Theorem}
\label{lin}
Assume that  the boundary  $\partial \Omega$  of  the domain is $C^2$.
Let be given 
$ (\overline{\Gamma}_i )_{1\leqslant  i \leqslant  d}  \in \R^d $,
$N  \geq 1$ distinct points $x_1 ,... ,x_N$   in   $\Omega$
 and $ (\alpha_l )_{1\leqslant  l \leqslant  N} \in \R^N$.
Let $T>0$ and  $z(t) :=(z_1 (t),...,z_N (t))$  be the unique solution (up to  the first
collision) in $C^ \omega (\lbrack 0,T \rbrack )$ of the Kirchoff-Routh-Lin equations of point vortices  (cf. Lemma \ref{def?}) with $x_l$ as initial positions, of respective strength $\alpha_l$,  for $1 \leqslant  l \leqslant  N$, with  $\overline{\Gamma}_i$ as respective circulation on the inner boundary  $C_i$, for $1 \leqslant  i \leqslant  d$.
Then  
\begin{eqnarray}
t \mapsto  \sum_{1 \leqslant l  \leqslant N}  \alpha_l \delta_{z_l (t) } ,
\end{eqnarray}
provides a weak solution  of the Euler equation on $\lbrack 0,T \rbrack$ (in the sense of Definition \ref{WV}) with $\omega_0 := \sum_{1 \leqslant l  \leqslant N}  \alpha_l \delta_{x_l} $ as initial data.  
\end{Theorem}
It should be argue that Theorem \ref{lin} belongs to the mathematical folklore. We provide an explicit proof in Appendix B for sake of completeness.
 We will show in  particular  in what sense  the motions of point vortices can be seen as  weak solutions of the Euler equations, adapting the weak vorticity formulation already used by Turkington \cite{turkington} (in a simply connected domain) and  Schochet \cite{schochet} (in the full plane) to  multiply connected domains.

In view of Theorem  \ref{start22} and Theorem  \ref{lin}, 
 it is natural to wonder to what extent it is possible to get rid of the boundary smoothness assumption.
The following result bridges theses two results showing that, 
for an initial data with a Yudovich vorticity (let say here bounded, in order to simplify the statement)  constant near the boundary, the smoothness of the flow map inside the domain can be obtained without assuming that the boundary is smooth.

\begin{Theorem} \label{Dmodifie}
Assume that the boundary  $\partial \Omega$ is $C^2$. Then  there exists $c >0$  such that for any  divergence free vector field $u_0$ in  $L^{2} (\Omega)$ tangent to the boundary $\partial \Omega$, with $\omega_0 := \curl u_0 \in L^{\infty}  ( \Omega  )$   constant outside of a compact $\underline{K} \subset \Omega $, for any  compact $K  \subset \Omega$, 
 the flow map $ \Phi$ is, for any $r \in (0,1)$, for any $T >0$, for any $M>1$,  Gevrey of order $M+2$  from $ \lbrack 0, T  \rbrack $ to $  C^{0 ,  r \exp(-cT  \| \omega_0 \|_{L^{\infty}  (\Omega )})} (K)$.
\end{Theorem}

Let us stress that there is a arbitrary small loss of Gevrey order with respect  to the result of Gamblin  \cite{gamblin} about Yudovich flows in the full plane (and also with respect to Theorem \ref{start2} when assuming that the boundary  $\partial \Omega$ is $C^\omega$). 

For classical flows, with vorticities constant near the boundary, it is possible to localize without any loss. Since this also  holds in three dimensions, we prefer to postpone this to Appendix C, in order to avoid any confusion about the setting of these results.

We also plan to investigate this issue of smoothness along the flow lines in the case where the vorticity of the flows has some  Dirac masses, in addition to a bounded (or ``slightly unbounded") part. This setting was introduced by Marchioro  and  Pulvirenti, see \cite{MP}.
Uniqueness is known to hold when the flow occupies the full plane, when the absolutely continuous part of the vorticity is bounded and when initially the point vortices are surrounded by regions of constant  vorticity, see also \cite{LM}. 
It is therefore natural to wonder if a strategy with a cut-off could allow to deal with this case, and for extensions to bounded domains, and to the case where the absolutely continuous part   is slightly unbounded. 
An underlying motivation is to prove some property of smoothness along the flow lines
for any setting where existence and uniqueness of   the incompressible Euler equations are known to hold.

%
%
%

 \section{Proof of Theorem \ref{start22}}
 \label{proof}
 
 This section is devoted to the proof of Theorem \ref{start22}.
 We will focus on the Gevrey case, the $C^\infty$ case would be a byproduct of the analysis.
We therefore assume that  the function $\rho(x) :=\dist(x,\partial \Omega)$ satisfies the following:   there exists $c_\rho > 1$ such  that for all $s \in \mathbb{N}$, on a neighborhood $\cW \subset \overline{\Omega}$ of the boundary $\partial \Omega$,
\begin{equation}
\label{bordGevrey}
\| \nabla^s \rho \| \leq c_\rho^s \, (s!)^M ,
\end{equation}
as a function (on $\cW$) with values in the set of symmetric $s$-linear forms.  

 We will proceed by regularization, working from now on a smooth flow, with the same notation. Since the estimates we are going now to get are uniform with respect to the regularization parameter, the result will follow. We refer to  \cite{gamblin} for more details on this step.

Let us also recall a few basic ingredients.
\begin{Definition}
A vector field $\cX$ from $\Omega$ to $ \R^2$ is said tangential harmonic if it is $W^{1,2} (\Omega;\mathbb{R}^2)$, satisfies $\div \cX =0$ and $\curl  \cX =0$ in $\Omega$, and $ \hat{n} \cdot \cX = 0$ on  $\partial \Omega$. 
\end{Definition}
Let us first recall the following classical result from the Hodge-De Rham theory.
\begin{Theorem}
\label{hodge}
The tangential harmonic  vector fields are smooth up to their boundary. Their
 set is a vector space  $H$ of dimension $d$, orthogonal, in $L^2(\Omega;\mathbb{R}^2)$, to any gradient of smooth functions. 
 There is a unique family $ \{ \cX_1 , ..., \cX_d  \}$ which are a basis of  $H$ and satisfy  $ \Gamma_i ( \cX_j ) = \delta_{i,j}$ for $1\leq i,j \leq d$. 
\end{Theorem}
It is a well-known result,  let us refer to \cite{MP}, Theorem $2.1$
 and to 
the appendix to introduction of \cite{katoana} for a detailed proof of the smoothness up to the boundary.
\begin{Definition}
We will denote by  $\Pi$ be the orthogonal projection of $L^2(\Omega;\mathbb{R}^2)$ onto the space $H$ of tangential harmonic vectors. 
\end{Definition}
\begin{Lemma}
\label{cpi}
There exists $C_\Pi > 0$  (depending only on $\Omega$) such that for any $p>2$, for any $f$ in $L^ p (\Omega)$,
 $ \|\Pi  f \|_{ L^{p}(\Omega) }  \leq C_\Pi  \| f \|_{ L^{p}(\Omega) }$.
\end{Lemma}
\begin{proof}
Let us orthonormalize the $ \cX_i $. We denote by $\tilde{\cX}_i$, $1\leq i \leq d$, the orthonormal system obtained. 
Then
$$ \Pi  f = \sum_{i=1}^d (\int_\Omega f \cdot \tilde{\cX}_i  )  \tilde{\cX}_i ,$$ 
and the $\tilde{\cX}_i$ are smooth.
\end{proof}
We will use the following elliptic regularity estimate.
\begin{Lemma}
\label{triv}
There exists $c,c_h >0$ such that  for any $p >  2$,
for any  smooth vector field $f$ from $\Omega$ to $ \R^2$ such that there exists $\phi$ in $W^{1,p}( \cW)$ such that
$( \hat{n} \cdot f ) |_{\partial \Omega} = \phi |_{\partial \Omega} $, 
\begin{equation}
\label{reg1}
\| f\|_{ W^{1,p} (\Omega) }  \leq c p
 ( \|\div  f \|_{ L^{p}(\Omega) } + \|\curl  f \|_{ L^{p}(\Omega) } + \|  \phi \|_{   W^{1,p}  ( \cW)  }  )+ c_h   \|\Pi  f \|_{ L^{p}(\Omega) } .
\end{equation}
\end{Lemma}
Lemma \ref{triv}  relies on Calderon-Zygmund theory of singular integral operators. 
A particular case   has been used  in Yudovich's proof of  Theorem  \ref{95} and of Theorem \ref{mod}  (cf.  \cite{yudo63},  \cite{yudo95}). 
The dependance on $p$ was crucial in his proof and it would also be crucial in the proof of  Theorem \ref{start22}. 
Since we did not find as it in the literature, 
we provide a proof for sake of completeness. We will use the Yudovich result: there exists $c >0$ such that  for any $p >  2$,
for any  smooth function $\varphi $ from $\Omega$ to $ \R$, satisfying  $\varphi =0$ on  $\partial \Omega$ or  $\partial_n \varphi =  0 $ on  $\partial \Omega$ and $ \int_{\Omega } \Delta  \varphi =0$,
\begin{equation}
\label{cz}
\|  \nabla \varphi \|_{ W^{1,p} (\Omega) }  \leq c p
  \|\Delta  \varphi \|_{ L^{p}(\Omega) }  .
\end{equation}
\begin{proof}
Let us still denote $\hat{n} $ a smooth extension  of the unit normal   supported in $\cW$.
There exists only one smooth function (up to a additive constant) $\varphi$ which satisfies $ \Delta \varphi = \div ( f - \phi \hat{n} )  $  in $\Omega$, and $\partial_n \varphi =  0 $ on  $\partial \Omega$. Using \eqref{cz} yields that 
there exists $c>0$ such that  for any $p >  2$,
\begin{equation*}
\label{reg1a}
\|  \nabla \varphi  \|_{ W^{1,p} (\Omega) }  \leq c p
 ( \|\div  f \|_{ L^{p}(\Omega) } + \| \phi  \|_{ W^{1,p} (\cW ) }  ).
\end{equation*}
There exists only one smooth function $\psi$ which satisfies $ \Delta \psi = \curl   ( f - \phi \hat{n} ) $  in $\Omega$, and $ \psi = 0 $ on  $\partial \Omega$.   Using \eqref{cz} yields that there exists $c>0$ such that  for any $p >  2$,
\begin{equation*}
\label{reg1b}
\|  \nabla \psi  \|_{ W^{1,p} (\Omega) }  \leq c p
 ( \| \curl   f \|_{ L^{p}(\Omega) } + \|  \phi \|_{   W^{1,p}  ( \cW)  }  )   .
\end{equation*}
Now let us observe that 
$$(Id- \Pi ) f =  \phi \hat{n} + \nabla \varphi + \nabla^\perp  \psi + \sum_{i=1}^d \beta_i \cX_i ,$$
 with $\beta_i := - {\| \cX_i \|^{-2}_{ L^{2}(\Omega) }}  \int_{ \Omega }  \phi \hat{n} \cdot  \cX_i   dx  $. Hence 
\begin{equation}
\label{reg1ii}
\|(Id- \Pi ) f\|_{ W^{1,p} (\Omega) }  \leq c p
 ( \|\div  f \|_{ L^{p}(\Omega) } + \|\curl  f \|_{ L^{p}(\Omega) } + \|  \phi \|_{   W^{1,p}  ( \cW)  }  ).
\end{equation}
It only remains to estimate the harmonic part.
It suffices to observe that 
$$
 \Pi  f = \sum_{i=1}^d (\int_\Omega  \Pi f \cdot \tilde{\cX}_i  )  \tilde{\cX}_i ,
 $$ 
 to get \eqref{reg1}, with a  constant $c_h$ (where $h$ stands for  harmonic)  which depends only on $\Omega$ (including through the $ \tilde{\cX}_i$), but not on $p$.
 \end{proof}
Another way to deal with the harmonic part is to consider the circulations, and for $1\leq i \leq d$ the function $ \phi_i $ in $C^{\infty}(\Omega)$ such that 
$\Delta  \phi_i = 0$ in $\Omega$,  with $\phi_i = \delta_{i,j}$, on $C_j$, for $j = 0, \ldots, d$. 
Let us recall the following (cf. \cite{lin}, \cite{koebe}).
\begin{Lemma}
\label{charlot}
For any smooth  vector field  $f$  from $\Omega$ to $ \R^2$, 
\begin{equation}
\label{charlot1}
 \Pi  f =  \sum_{i=1}^d  \alpha_i (f)  \cX_i ,
\end{equation}
where
\begin{equation}
\label{charlot2}
 \alpha_i (f) :=  \int_\Omega   \phi_i \curl f  +  \Gamma_i (f)   .
\end{equation}
\end{Lemma}
\begin{proof}
Thanks to Green's identity we get for  $1\leq i \leq d$,
\begin{equation*}
 \alpha_i (f) = -  \int_\Omega   \nabla^\perp \phi_i \cdot f .
 \end{equation*}
In particular, this yields for $1\leq i,j \leq d$,
\begin{equation}
\label{charlot4}
 \int_\Omega   \nabla^\perp \phi_i \cdot \cX_j = -  \delta_{ij} .
 \end{equation}
Moreover  $\nabla^\perp \phi_i \in H$, for  $1\leq i \leq d$, so that 
\begin{equation}
\label{charlot3}
 \alpha_i (f) = -  \int_\Omega    \nabla^\perp \phi_i \cdot  \Pi  f .
 \end{equation}
Now let us look lor the coefficients $ \beta_i (f)$ such that $ \Pi  f =  \sum_{j=1}^d  \beta_j (f)  \cX_j $.  Plugging  this into \eqref{charlot3} and taking into account  \eqref{charlot4} we get   $ \beta_i (f) =  \alpha_i (f)$, for  $1\leq i \leq d$, and therefore \eqref{charlot2}.
 \end{proof}
In the sequel we will  need the  following consequence  of Lemma \ref{triv} and Lemma \ref{charlot}.
\begin{Lemma}
\label{triv2}
 There exists $c >0$  (depending only on $\Omega$) such that  for any $p > 2$,
for any smooth divergence free vector field  $f$ tangent to the boundary, there holds 
\begin{equation}
\label{reg2} \| f\|_{ W^{1,p} (\Omega) } \leq c p
  \|\curl  f \|_{ L^{p}(\Omega) } + c  \sum_{i=1}^d  | \Gamma_i (f)  |.
\end{equation}
\end{Lemma}
\begin{proof}
Thanks to Lemma \ref{charlot} there exists $c >0$ such that   for any $p > 2$,
$$  \| \Pi  f \|_{ L^{p}(\Omega) }  \leq c  \|\curl  f \|_{ L^{p}(\Omega) } + c  \sum_{i=1}^d  | \Gamma_i (f)  |.$$
Plugging this in  \eqref{reg1} therefore yields \eqref{reg2}.
\end{proof}

It is also useful to have in mind the following form of the H\"older inequality: for any integer $k$, for any $ \theta := (s, \alpha)$ in
\begin{equation*}
\mathcal{A}_{k} := \{ \theta \in  \N^* \times  \N^s / \ 2 \leqslant s  \leqslant k+1  \text{ and }
\alpha := ( \alpha_1,\ldots, \alpha_s )  \in \N^s / \ | \alpha |    = k+1 - s \},
\end{equation*}
where the notation $ | \alpha |$ stands for $| \alpha | := \alpha_1 + \ldots+  \alpha_s$,
 and for any $p \geqslant 1$, 
 \begin{eqnarray}
  \label{holder}
\| \prod_{i=1}^{s}  f_i  \|_{ L^{  \frac{p}{k+1} } (\Omega) }  \leqslant  \prod_{i=1}^{s} \|  f_i  \|_{ L^{ \frac{p}{ \alpha_i +1} } (\Omega) } .
\end{eqnarray}
\ \par
\noindent
 We will use some  formal identities, obtained in \cite{ogfstt}, of  the iterated material derivatives $ (D^k u)_{ k \in \N^* }$, where 
\begin{equation*}
D := \partial_{t} + u.\nabla .
\end{equation*}
We use the following notations: for $\alpha := ( \alpha_1,\ldots, \alpha_s )  \in \N^s$ we will denote $\alpha ! :=  \alpha_1 ! \ldots \alpha_s !$. We  denote by $\trace\{A\} $ the trace 
of $A \in {\mathcal M}_{3}(\R)$ and by $\as\{A\}:=A-A^*$ the antisymmetric part of $A \in {\mathcal M}_{3}(\R)$.
\begin{Lemma}\label{P1}
For $k \in \N^*$, we have in $\Omega$
\begin{eqnarray}\label{P1fNew}
\div D^k u=\trace\left\{F^k [u]\right\} \text{ where } 
F^k [u] := \sum_{\theta   \in \mathcal{A}_{k}  } c^1_k (\theta  )  \,  f(\theta)  [u], \\ 
\label{P2fNew}
\curl D^k u=\as\left\{G^k [u]\right\} \text{ where } 
G^k [u] :=  \sum_{\theta  \in \mathcal{A}_{k}  }  c^2_k (\theta )   \, f(\theta)  [u],
\end{eqnarray}
 and on the boundary $ \partial \Omega$
\begin{eqnarray}\label{P4fNew}
\hat{n} \cdot D^k u =H^k [u] \text{ where } 
H^k [u] :=  \sum_{ \theta  \in \mathcal{A}_{k}  }  c^3_k (\theta )  \, h(\theta)  [u],
\end{eqnarray}
where
\begin{gather} \label{DefsFetH}
f( \theta)  [u] : = \nabla D^{\alpha_1} u \cdot \ldots \cdot \nabla D^{\alpha_s} u
\text{ and } h( \theta)  [u] : =\nabla^s \rho \{ D^{\alpha_1} u , \ldots  ,D^{\alpha_s} u \},
\end{gather}
and where, for $i=1$, $2$, the $c^i_k (\theta )$ are integers satisfying  $ |c^i_k ( \theta) | \leqslant \frac{k ! }{\alpha ! }$, and the $ c^3_k (\theta)  $ are negative integers satisfying  $ | c^3_k (\theta)  | \leqslant \frac{k ! }{\alpha ! (s-1)!}$.
\end{Lemma}

Concerning the pressure it is possible to get by induction  from (\ref{2DincEuler}) the following identities (cf. \cite[Prop. 3.5]{katoana}).
\begin{Lemma}\label{pressure}
For $k\geq 1$, we have in  $\Omega$
\begin{equation}
D^k u +\nabla D^{k-1} p= K^k[u] 
	\label{t3.4}
\end{equation}
where $K^{1}[u]=0$ and for $k\geq 2$,
\begin{equation}
\label{ttt}
K^{k}[u]=-\sum_{s=1}^{k-1}  \dbinom{k-1}{ s} \nabla D^{ s-1} u \cdot D^{k-s} u.
\end{equation}
\end{Lemma}

\subsection{Walking down the scale, slowly}
 
 Let 
\begin{equation} \label{nota}
k_0 \in \N^*, \ r \in (0,1), \   p_1  \geq \frac{ 2 }{1-r }
  \text{ and } p_2   \geq  p_1 \cdot (k_0 + 1) .
  \end{equation}
 Let us introduce for $L>0$ the following function
\begin{equation} \label{Eq:DefGammaL}
\gamma(L) := L^{-1} \sup_{k  \geq 1} \left\{ 3 \sum_{s=2}^{k+1}  \ s^M  L^{2-s}  \, c_\rho^s \left(\frac{ k+1 }{ k-s+2 }\right)^2 20^{s} \, + \, C_\Pi   \sum_{s=1}^{k-1}   (\frac{k-s }{k s})^M \left(\frac{ k+1}{ (k-s+1) s } \right)^2 \right\} ,
\end{equation}
where  $c_\rho $ (respectively  $C_\Pi > 0$) is the constant introduced in 
\eqref{bordGevrey}  (resp. Lemma \ref{cpi}).
We fix $L$ large enough such that  
\begin{eqnarray}
\label{le} \gamma(L) \leq \frac{1 }{c_{\Omega}} ,
 \end{eqnarray}
where  $c_{\Omega}$ is the maximum of $1$ and of the constants  $c$ and $c_h$ introduced in Lemma \ref{triv}.
 We are going to prove recursively that for any integer $k  \leqslant  k_0$, 
 \begin{eqnarray}
  \label{proofit}
  \| D^k u \|_{ W^{1, \frac{p_2}{k+1}} (\Omega) }  
   \leqslant  p_2^k \frac{ (k!)^M L^k }{(k+1)^2} \|u\|^{k+1}_{ W^{1,p_2}   (\Omega) }   ,
\end{eqnarray}

 For $k=0$ there is nothing to prove.
  Now let us assume that Eq. \eqref{proofit} is proved up to $  k -1 \leqslant  k_0 -1$. 
  
\subsubsection{Estimate of $F^k  [u]$ and $G^k  [u]$}
 
 Applying the H\"older inequality \eqref{holder} to the definition of  $f( \theta)  [u]$ in
 \eqref{DefsFetH},  for $\theta  \in \mathcal{A}_{k} $,   yields that 
\begin{equation*}
 \| f( \theta)  [u]  \|_{ L^{ \frac{p_2}{k+1}} (\Omega) }
 \leq  \prod_{i=1}^s  \| D^{\alpha_i } u \|_{ W^{1,\frac{p_2}{ \alpha_i +1}}   (\Omega) } .
\end{equation*}
Using  the induction hypothesis and since for  $\theta  \in \mathcal{A}_{k} $, $| \alpha | = k+1 - s$, we have
\begin{equation*}
 \| f( \theta)  [u]  \|_{ L^{ \frac{p_2}{k+1}} (\Omega) }
 \leq L^k  \,  \| u \|^{k+1}_{ W^{1,p_2}   (\Omega) }  \, (\alpha !)^M
 L^{1-s}  p_2^{| \alpha | }   \prod_{i=1}^s  \frac{1}{ (1+\alpha_i )^2}.
\end{equation*}
Now thanks to Lemma \ref{P1},   we obtain
\begin{eqnarray*}
\| F^k  [u]  \|_{ L^{ \frac{p_2}{k+1}} (\Omega) }
 \leq    k! L^k  \,  \| u \|^{k+1}_{ W^{1,p_2}   (\Omega) }  \,  \sum_{s=2}^{k+1} \ L^{1-s}  \, 
\sum_{ \alpha  /  \, | \alpha | = k+1 - s }  \,   (\alpha !)^{M-1} p_2^{ | \alpha |} \prod_{i=1}^s  \frac{1}{ (1+\alpha_i )^2} .
\end{eqnarray*}
When  $\theta  \in \mathcal{A}_{k} $, $2  \leq s \leq  k+1$ and $| \alpha | = k+1 - s$, then $| \alpha |  \leq k-1$ so that 
\begin{eqnarray}\label{new0.0}
\| F^k  [u]  \|_{ L^{ \frac{p_2}{k+1}} (\Omega) }
 \leq  p_2^{k-1}  (k!)^M L^k  \,  \| u \|^{k+1}_{ W^{1,p_2}   (\Omega) }  \,  \sum_{s=2}^{k+1} \  \frac{L^{1-s}}{k^{M-1}}  \, 
\sum_{ \alpha  /  \, | \alpha | = k+1 - s }  \,  \prod_{i=1}^s  \frac{1}{ (1+\alpha_i )^2}  .
\end{eqnarray}
We now use the following combinatorial lemma (cf. \cite[Lemma 7.3.3]{cheminsmf}).  
\begin{Lemma} \label{LemmeCheminSMF}
For any couple of positive integers $(s,m)$ we have
\begin{equation}\label{DefUpsilon}
\sum_{\substack{{\alpha \in \N^{s}} \\ {|\alpha|=m} }}  \Upsilon(s,\alpha) \leq \frac{20^{s}}{(m+1)^{2}}, \text{ where }
\Upsilon(s,\alpha):=\prod_{i=1}^s \frac{1}{(1+\alpha_{i})^2}.
\end{equation}
\end{Lemma}
We deduce from \eqref{new0.0} and from the above lemma that
\begin{eqnarray} \label{DeCadix}
 \| F^k  [u]  \|_{ L^{ \frac{p_2}{k+1}} (\Omega) }
    \leq  p_2^{k-1} \frac{ (k!)^M L^k }{(k+1)^2}  \| u \|^{k+1}_{ W^{1,p_2}   (\Omega) } \,  \sum_{s=2}^{k+1} \  \frac{L^{1-s}}{k^{M-1}}   \,  20^{s}   \, \frac{(k+1)^2}{ (k-s+2 )^2}.
\end{eqnarray}
We have the same bound on $  \| G^k[u]  \|_{ L^{ \frac{p_2}{k+1}} (\Omega) }$ using \eqref{P2fNew} instead of \eqref{P1fNew}. 

  \subsubsection{Estimate of $H^k  [u]$}
 
 Applying  \eqref{holder} and using 
\eqref{bordGevrey}  yields, for $\theta  \in \mathcal{A}_{k} $, that 
\begin{equation*}
 \| h( \theta)  [u]  \|_{  W^{1, \frac{p_2}{k+1}} (\cW ) }
 \leq c_\rho^s   \,    (s!)^M  \prod_{i=1}^s  \|  D^{\alpha_i} u \|_{ W^{1,\frac{p_2}{ \alpha_i +1}}   (\Omega) } .
\end{equation*}
By using  the induction hypothesis, we have
\begin{equation*}
 \| h( \theta)  [u]  \|_{  W^{1, \frac{p_2}{k+1}} (\cW ) }
 \leq L^k  \,  \| u \|^{k+1}_{ W^{1,p_2}   (\Omega) }  \, (\alpha !)^M  (s!)^M
 L^{1-s} c_\rho^s   \,    p_2^{| \alpha | }   \prod_{i=1}^s  \frac{1}{ (1+\alpha_i )^2}.
\end{equation*}
Thanks to Lemma \ref{P1}  and Lemma \ref{LemmeCheminSMF} we obtain
\begin{eqnarray}
\label{new1.4}
\| H^k  [u] \|_{W^{1,\frac{p_2}{k+1}} (\cW)} \leq  p_2^{k-1}
\frac{ (k!)^M L^k }{(k+1)^2}  \| u \|^{k+1}_{ W^{1,p_2}  (\Omega ) }  \sum_{s=2}^{k+1} \ s^M  L^{1-s}  \, c_\rho^s   \,  20^{s} \,   \frac{(k+1)^2}{ (k-s+2 )^2}  .
\end{eqnarray}
 \subsubsection{Estimate of $K^k  [u]$}
 
  Applying the H\"older inequality \eqref{holder} to the definition \eqref{ttt} of $K^k  [u]$
  yields for any $k \geqslant 2$, 
 \begin{eqnarray*}
 \|  K^k  [u]  \|_{ L^{ \frac{p_2}{k+1} } (\Omega) }  \leqslant 
\sum_{s=1}^{k-1} \  \dbinom{k-1}{s}  \| D^{s-1} u \|_{ W^{1,\frac{p_2}{s} }  (\Omega)}
\| D^{k-s} u \|_{ W^{1,\frac{p_2}{k+1-s} } (\Omega)}  .
 \end{eqnarray*}
  By using  the induction hypothesis   we get 
 \begin{eqnarray*}
 \|  K^k  [u]  \|_{ L^{ \frac{p_2}{k+1} } (\Omega) } 
  \leqslant  p_2^{k-1} \frac{ (k!)^M L^k }{(k+1)^2}  \| u \|^{k+1}_{ W^{1,p_2}   (\Omega) } \,  L^{-1} \sum_{s=1}^{k-1} \  (\frac{k-s }{k s})^M  \left( \frac{k+1 }{s(k-s+1)}\right)^2 .
 \end{eqnarray*}
Finally using Lemma  \ref{pressure}  and  Theorem  \ref{hodge} we have   $ \Pi D^k u =  \Pi  K^k [u] $ so that, thanks to Lemma \ref{cpi}, we get
 \begin{eqnarray}
 \label{pi}
\|  \Pi   D^k u   \|_{ L^{ \frac{p_2}{k+1} } (\Omega) }  \leqslant   C_\Pi p_2^{k-1} \frac{ (k!)^M L^k }{(k+1)^2}  \| u \|^{k+1}_{ W^{1,p_2}   (\Omega) } \,  L^{-1}
\sum_{s=1}^{k-1} \   (\frac{k-s }{k s})^M  \left( \frac{k+1 }{s(k-s+1)}\right)^2 .
 \end{eqnarray}

  \subsubsection{Conclusion}
 
 We now apply Lemma  \ref{triv} to $f=  D^k u$ (observing that, thanks to \eqref{nota}, we have  $ \frac{ p_2 }{ k+1}  > 2$) and we use \eqref{Eq:DefGammaL}-\eqref{le}-\eqref{DeCadix}-\eqref{new1.4}-\eqref{pi} to get  Eq. \eqref{proofit} at rank $k$.
 
  \subsection{Walking down the scale, from high enough}
  \label{accord}
We now apply  Eq. \eqref{proofit}  with  $p_2 (k+1)$ instead of $p_2$, and we use Stirling's formula to obtain that for any $k \in \N$, for any $p_2  \geq \frac{2}{1-r}$,
 \begin{eqnarray}
  \label{proofit3}
    \| D^k u \|_{ W^{1,p_2 } (\Omega)  }  
       \leqslant  (p_2 (k+1))^k \frac{(k!)^{M} L^k }{(k+1)^2}   \| u \|^{k+1}_{ W^{1,p_2  (k+1)}   (\Omega) }     \leqslant   p_2^k (k!)^{M+1} \tilde{L}^k    \| u \|^{k+1}_{ W^{1,p_2  (k+1)}   (\Omega) }   .
   \end{eqnarray}
 So far time has intervened only as a parameter, and the inequality  \eqref{proofit3} holds for any time. We will now estimate its right hand side with respect to the initial data.
 First thanks to  \eqref{reg2}
 there exists $c >0$ such that  for any $k$,
\begin{equation*}
\| u \|_{ W^{1,p_2  (k+1)} (\Omega) } \leq c p_2  (k+1)
  \|\curl  u \|_{ L^{p_2  (k+1)}(\Omega) }  + c  \sum_{i=1}^d  | \Gamma_i ( u)  |.
\end{equation*}
Now conservation of the $L^p$ norms of the vorticity and Kelvin's circulation theorem yields
\begin{eqnarray*}
\| u \|_{ W^{1,p_2  (k+1)} (\Omega) } &\leq& c p_2  (k+1)
  \| \omega_0 \|_{ L^{p_2  (k+1)}(\Omega) }  + c  \sum_{i=1}^d  | \Gamma_i ( u_0)  | ,
  \\ &\leq& c p_2  (k+1) \theta( p_2  (k+1) )  + c  \sum_{i=1}^d  | \Gamma_i ( u_0)  | ,
\end{eqnarray*}
since $u_0$ in  $\bY_\theta$.
Plugging this into  \eqref{proofit3} and using again Stirling's formula, we obtain that  there exists $L >0$ depending only on $\Omega$ such that  for any $k$,
 \begin{eqnarray*}
  \| D^k u \|_{ W^{1,p_2 } (\Omega)  }
       \leqslant   p_2^k (k!)^{M+1} L^{k+1}   \big(k!  p_2^{k+1} \theta( p_2  (k+1) )^{k+1}  + \sum_{i=1}^d  | \Gamma_i ( u_0)  |^{k+1} \big) .
   \end{eqnarray*}
Thanks to Morrey's inequality, there exists $C>0$ such that for any smooth function $u$ on $\overline{\Omega}$, for any $r \in (0,1)$,
$  \| f \|_{ C^{0,r} (\Omega)  }
       \leqslant  C  \| f \|_{ W^{1,p_2 } (\Omega)  } $,
 where $p_2 = 2/(1-r)$.
 This allows to bound $ \| D^k u \|_{ C^{0,r} (\Omega) }$ thanks to $ \| D^k u \|_{ W^{1,p_2 } (\Omega)  }$. Then we
differentiate Eq. \eqref{flow} to get  $\partial^{k +1 }_t \Phi  (t,x)= D^k u (t, \Phi (t,x)).$
We consider $T>0$ and we use  Lemma \ref{trivcomp}, 
  and the proof of Theorem \ref{start22} is over.

 \section{Proof of Theorem \ref{localreg}}

 In order to  obtain the propagation of local smoothness  we will use some 
interior elliptic regularity, instead of Lemma  \ref{triv} and  \ref{triv2}. 
Let us first recall the following Schauder estimate  (cf. \cite{GT}).
\begin{Lemma}
\label{SchauderHolder}
Let $\mathcal{D}$ be an open set such that  $\overline{\mathcal{D}}  \subset \Omega$.
Let $u$ be a continuous vector field on $\overline{\Omega}$  such that $\div u = 0 $ in $ \Omega$,  $ \hat{n} \cdot u = 0$ on  $\partial \Omega$ and  $\curl  u |_{\mathcal{D} } $ in $C^{\lambda ,r}_{\text{loc}} ({\mathcal{D}})$. Then  $u$ is in $C^{\lambda +1 ,r}_{\text{loc}} ({\mathcal{D}})$.
\end{Lemma}
Lemma  \ref{SchauderHolder}  extends as follows to Dini continuous vorticities.
\begin{Lemma}
\label{SchauderDini}
Let $\mathcal{D}$ be an open set such that  $\overline{\mathcal{D}}  \subset \Omega$.
Let $u$ be a continuous vector field on $\overline{\Omega}$ such that $\div u = 0 $ in $ \Omega$,  $ \hat{n} \cdot u = 0$ on  $\partial \Omega$ and  $\curl  u |_{\mathcal{D} } $
 in $C_{\mu,\text{loc}} (\mathcal{D})$,  with $\mu$ is a Dini modulus of continuity. Then $u$ is in  $C^1 ({\mathcal{D}})$.
\end{Lemma}
Above the space $C_{\mu,\text{loc}} (\Omega_0 )$ denotes the set of 
  functions  which are in  $C_{\mu} (K)$ for any compact subset $K \subset \Omega_0 $.
 We refer to  the paper  \cite{koch} for a closer statement.
  We provide a proof for sake of completeness.
 \begin{proof}
 Let us introduce  $  \Psi_0$ the unique solution of   $ \Delta  \Psi_0  = \curl u   $ in $\Omega$ and $  \Psi_0  = 0$ on  $\partial \Omega$. 
 We then denote  $ v = \nabla^\perp \Psi_0 $ and observe that $u-v $ is in the space $H$ of tangential harmonic vector fields of $\Omega$.
 Next we introduce 
 $$ \Psi (x):=  \frac{1}{2\pi} \int_{\mathcal{D} }  \log   \| x-y \| \cdot  (\curl  u) (y) dy ,$$
 which is  in  $C^2 ({\mathcal{D}})$ and satisfies  $ \Delta  \Psi  =  \curl  u $ in $\mathcal{D}$, according to Lemma \ref{SchauderDiniPo} in Appendix A. It suffices to observe that $ \Psi- \Psi_0 $ is harmonic in $\mathcal{D}$  to conclude the proof.
  \end{proof}
We are now equipped to start the proof of  Theorem  \ref{localreg}.
Let us recall that we assume that  the boundary  $\partial \Omega$  of  the domain is $C^2$,  that the initial data $u_0$ is in  $\bY_{\theta_m }$, with   $m \in \N^*$, that $\Omega_0$ is a  open subset such that 
 $\overline{ \Omega_0 } \subset  \Omega $ and  that  $\omega_0 |_{ {  \Omega_0}}$ is in $C^{\lambda,r}_{\text{loc}} (  \Omega_0 )$ with $\lambda$ in $\N$ and $r \in (0,1)$. 
Since  $ \Phi(t,.)^{-1 } $ satisfies \eqref{flow} with $-u$ instead of $u$, 
arguing as in Lemma \ref{Gamma}, we have that $ \Phi(t,.)^{-1 } $ is in  $C^0 (\lbrack 0, T  \rbrack, C_{\Gamma_t} (\Omega))$. 
Let us denote $\Omega_t := \{  \Phi(t,x) , \/ \ x \in \Omega_0 \} $.
Proceeding as in Lemma \ref {trivcomp}, we get that $\omega (t,\cdot) := \omega_0 (\Phi(t,.)^{-1 } )
 \in C^0 (\lbrack 0, T  \rbrack, C_{\Gamma_t^r   ,\text{loc}} ( {\Omega}_t  ))$ (this notation is slightly improper but does not lead here to any confusion). Now thanks to  Lemma \ref{mdini}, the modulus of continuity $\Gamma_t^r$ is Dini. Applying  Lemma \ref{SchauderDini} yields that  $u $ is in  $C^0 (\lbrack 0, T  \rbrack, C^1 (\Omega_t ))$.
By integration, we infer that $ \Phi$ and  $t \mapsto \Phi(t,.)^{-1 } $ are in  $C^0 (\lbrack 0, T  \rbrack, C^1 (\Omega_0 ))$. Proceeding again as in Lemma \ref {trivcomp}, we get that $\omega (t,\cdot) \in C^{0,r}_{\text{loc}}  ( {\Omega}_t  )$. Then Lemma  \ref{SchauderHolder} yields that  $u (t,\cdot)$ is in  $ C^{1,r}_{\text{loc}}  (\Omega_t )$. We can now repeat the bootstrapping arguments exactly as in Proposition $8.3$ of \cite{BM} to end the proof.

 \section{Proof of Theorem \ref{start22localh}}

This section is devoted to the proof of Theorem \ref{start22localh}.
Let us start by repeating two preliminary remarks of the proof  of Theorem \ref{start22}. First we will focus on the Gevrey case, the $C^\infty$ case being a byproduct of the analysis. Secondly we will work from now on a smooth flow, the result following by a classical regularization argument.

\subsection{Shrinking the compact, slowly}
 
 Let us first introduce a notation: when $K$ is a compact and $\eps > 0$ we denote $K_\eps := \{ x \in K / \   \text{ dist }  (x,K^c ) \geqslant   \eps  \} $. 
  Let us also introduce for $L>0$ the following function
\begin{equation} \label{Eq:DefGammaL2}
\tilde{\gamma} (L) := 2 L^{-1} \sup_{k  \geq 1}  \sum_{s=2}^{k+1}  \ k^{-M}  L^{2-s}  \, \left(\frac{ k+1 }{ k-s+2 }\right)^2 20^{s}  .
\end{equation}
We fix $L$ large enough (depending on $r$) such that for any integer $k  \geqslant  1$,
 \begin{eqnarray}
  \label{proofit333}
    \| D^k u \|_{ C^{0,r } (\Omega)  }  
       \leqslant c_\Omega \ {\gamma}(L) \frac{  (k!)^{M+1} {L}^k  }{(k+1)^2}   \| u \|^{k+1}_{ W^{1,p_2  (k+1)}   (\Omega) }   
   \end{eqnarray}
with $p_2 := \frac{ 2 }{1-r }$ (what is possible according to the analysis of section \ref{accord})  and such that 
\begin{eqnarray}
\label{le2}   c_\Omega \, c^{ \lambda_0 + 1 }_{i} (\gamma (L)  +   \tilde{\gamma}(L) ) \leq  1 ,
 \end{eqnarray}
where  $c_{i}$ will appear  in Lemma \ref{LemmeInterior}.

 Let  $k_0 \in \N^*$ and $\eps > 0$ such that $  \text{ diam } K > k_0 \eps$.
We are going to prove recursively for any integer $ 1  \leqslant \lambda  \leqslant  \lambda_0 + 1$,  and then recursively for any integer $k $ such that $1  \leqslant k \leqslant  k_0$ that
 \begin{eqnarray}
  \label{decrea}
  \| D^k u \|_{ C^{ \lambda,r} (K_{k \eps } ) }  
   \leqslant     c_\Omega \, c^{ \lambda  }_{i} (\gamma (L)  +   \tilde{\gamma}(L) ) 
 \frac{ (k!)^{M+1}  L^k \eps^{-k \lambda (1 +r)}}{(k+1)^2} 
 N_{\lambda , r , K  }^{k+1} ,
\end{eqnarray}
 where 
 $$
 N_{\lambda , r , K  }  := \|u\|_{   C^{\lambda ,r} (K  ) }    +    \|u\|_{ W^{1,p_2 (k+1)}   (\Omega) }.$$
 
 Let us assume that Eq. \eqref{decrea} is proved up to $  k -1 \leqslant  k_0 -1$.

Looking forward to the definition of  $f( \theta)  [u]$ in
 \eqref{DefsFetH}, we have that, for $\theta  \in \mathcal{A}_{k} $,  
\begin{equation*}
 \| f( \theta)  [u]  \|_{ C^{\lambda -1 ,r} (K_{(k-1) \eps } ) }
 \leq  \prod_{i=1}^s  \| D^{\alpha_i } u \|_{  C^{\lambda,r} (K_{(k-1) \eps } ) } 
 \leq  \prod_{i=1}^s  \| D^{\alpha_i } u \|_{ C^{\lambda,r} (K_{ \alpha_i  \eps } )  } .
\end{equation*}

Using  the induction hypothesis, that \eqref{le2} and $c_i > 1$ imply $ c_\Omega \, c^{ \lambda }_{i} (\gamma (L)  +   \tilde{\gamma}(L) ) \leq  1 $, 
we therefore obtain:
\begin{equation*}
 \| f( \theta)  [u]  \|_{ C^{\lambda -1 ,r} (K_{(k-1) \eps } ) }
 \leq    \eps^{- (k-1)  \lambda  ( 1 +r)}  L^k    N_{\lambda , r , K  }^{k+1} ( \alpha !)^{M+1} L^{1-s}  \prod_{i=1}^s  \frac{1}{ (1+\alpha_i )^2}.
\end{equation*}
Now thanks to Lemma \ref{P1},   we obtain
\begin{eqnarray*}
\| F^k  [u]  \|_{ C^{\lambda -1 ,r} (K_{(k-1) \eps } ) }
 \leq  
(k!)^{M+1}   \eps^{-(k-1)  \lambda (1 +r) } L^k  N_{\lambda , r , K  }^{k+1}  \,  \sum_{s=2}^{k+1} \,  \frac{L^{1-s}}{k^{M}}  
\sum_{ \alpha  /  \, | \alpha | = k+1 - s }  \,    \prod_{i=1}^s  \frac{1}{ (1+\alpha_i )^2}  .
\end{eqnarray*}
 Using  Lemma \ref{LemmeCheminSMF} we obtain
\begin{eqnarray*}
\| F^k  [u]  \|_{ C^{\lambda -1 ,r} (K_{(k-1) \eps } ) }
 \leq  \frac{ (k!)^{M+1} L^k }{(k+1)^2} 
   \eps^{-(k-1) \lambda (1 +r)  }   N_{\lambda , r , K  }^{k+1}  \,  \sum_{s=2}^{k+1} \,  \frac{L^{1-s}}{k^{M}}  
    \,  20^{s}   \, \frac{(k+1)^2}{ (k-s+2 )^2}.
\end{eqnarray*}
 We have the same bound on $  \| G^k[u]  \|_{  C^{0,r} (K_{(k-1) \eps } )  }$. 
 
 In order to obtain  \eqref{decrea} it then suffices to apply the following lemma to $f = D^k u $ and   $\tilde{ \eps } = (k-1) \eps  $ using  that \eqref{le2} implies $ c_{i}    \tilde{\gamma}(L)  \leq  1 $, 
 the inequality  \eqref{le2}, and the inequality  \eqref{proofit333} (respectively  the inequality \eqref{decrea} with $\lambda-1$ instead of $\lambda$ ) if $ \lambda = 1$ (resp. if $ \lambda > 1$).

\begin{Lemma}
\label{LemmeInterior}
 There exists $c_{i} > 1$ such that for any $\eps , \tilde{\eps}$ in $(0,1)$, for any $f  \in C^{ \lambda - 1,r} (K_{ \tilde{\eps} } )$, such that $\div f$ and $\curl  f$ are also  in $C^{ \lambda -1,r} (K_{ \tilde{\eps} } )$, then 
 $f  \in C^{ \lambda,r} (K_{\eps + \tilde{\eps} })$ and 
 \begin{eqnarray}
  \label{int1}
  \| f \|_{  C^{\lambda ,r} (K_{\eps + \tilde{\eps} } ) }  
   \leqslant  
c_{i} \, \eps^{-(1+r)} (  \| f \|_{  C^{\lambda -1,r} (K_{\tilde{\eps}} )  }  +  \| \curl  f \|_{  C^{\lambda -1,r} (K_{\tilde{\eps}} )  } 
+  \| \div  f \|_{  C^{\lambda-1,r} (K_{\tilde{\eps}} )}) .
\end{eqnarray}
\end{Lemma}
\begin{proof}
 Let us first recall that there exists $C_1 > 0$, which only depends on $r$, such that for any $v $ in $C^{0,r} (\R^2 )$ such that $\div v$ and $\curl  v$ are also  in $C^{0,r} (\R^2)$, then $v  \in C^{1 ,r} (\R^2)$ and 
 \begin{eqnarray}
  \label{int2}
  \| v \|_{  C^{1,r} ( \R^2 ) }  
   \leqslant  
C_1  (  \| v \|_{  C^{0,r} (\R^2 )  }  +  \| \curl  v \|_{  C^{0,r} (\R^2 )  } 
+  \| \div  v \|_{  C^{0,r} (\R^2 )  } ) .
\end{eqnarray}
 On the other hand there exists $C_2 > 0$, which only depends on $r$, such that  for any $\eps , \tilde{\eps} \in (0,1)$,
 there exists $\phi \in C^{\infty} (\R^2 )  $ such that $\phi |_{  K_{\tilde{\eps}}^c  } = 0$ and $\phi |_{  K_{\eps + \tilde{\eps} }  } = 1$ and $  \| \phi \|_{  C^{1 ,r} ( \R^2 ) }    \leqslant   C_2  \,  \eps^{-(1 +r)}$. Thus it is sufficient to apply \eqref{int2} to the function $v := \phi \, \partial^\alpha f$, for $| \alpha |= \lambda-1$ to conclude.
\end{proof}

\subsection{Shrinking the compact, from slightly larger}
 
Let $K$ be a compact subset of  $\Omega_0$.
Let $\underline{K}$ be a compact set such that  $K  \subset \dot{\underline{K}}  \subset \underline{K} \subset {\Omega_0}$. 
Then we  apply \eqref{decrea} with   $ \lambda  = \lambda_0 + 1$, $\eps :=  \text{ dist } (K, \underline{K}^c ) / k $, with $ \underline{K}$ instead of $K$, and using the inequality  \eqref{le2},   to obtain that  there exists $L>0$ such that for any integer $k  \in \N$, 
 \begin{eqnarray}
  \label{reastuce}
  \| D^k u \|_{ C^{\lambda_0 + 1 ,r} (K ) }  
   \leqslant   
  (k!)^{M+1+ (\lambda_0 + 1) (1 +r)}  L^k N_{\lambda_0 + 1 , r ,\underline{K}}^{k+1}  .
\end{eqnarray}

We then conclude as in the proof of  Theorem \ref{start22}.

 \section{Proof of Theorem  \ref{Dmodifie} }
 \label{sDmodifie}

Let $M>1$ and $T>0$.  There exists a compact $K'$ such that $K \subset  K'  \subset \Omega$ such that for any $t$ in $ \lbrack 0,T \rbrack$, the vorticity is constant outside of  $K'$. 
There exists $\chi:  \Omega  \rightarrow \lbrack 0,1 \rbrack$ Gevrey of order $M$ which vanishes in a neighborhood of the boundary  $\partial \Omega$ and which is equal to one on $K'$. 
We consider the vector field
\begin{equation*}
D_\chi := \partial_{t} + \chi u.\nabla .
\end{equation*}
The idea is then to proceed as in section \ref{proof} estimating recursively the $D_\chi^k u$, for $k$ in $\N^*$, instead of the $D^k u$.
The motivation for introducing the cut-off $\chi$ is that the  identity \eqref{P4fNew} becomes: on the boundary $ \partial \Omega$, for $k$ in $\N^*$, 
$\hat{n} \cdot D_\chi^k u = 0 $.
Moreover the circulations  $ \Gamma_i ( D_\chi^k u)$, for $1\leq i \leq d$ and $k \geq 1$,  vanish.
The $\div D_\chi^k u $ and the  $\curl  D_\chi^k u $ can be estimated thanks to the $D_\chi^j u$, with $j < k$, and with some extra factors involving $\chi$ and its derivatives, by using the following identities:
\begin{eqnarray}
	D_\chi (\psi_1\psi_2)=(D_\chi \psi_1)\psi_2+\psi_1(D_\chi \psi_2) \label{t3.0},\\
	\nabla (D_\chi \psi) -D_\chi (\nabla \psi) = (\nabla (\chi u)) \cdot (\nabla \psi), \label{t3.1}\\
	\div D_\chi \psi - D_\chi \div \psi = \trace \left\{(\nabla  (\chi u))\cdot(\nabla \psi) \right\}, \label{t3.2}\\
	\curl D_\chi \psi - D_\chi \curl \psi = \as \left\{(\nabla  (\chi u))\cdot(\nabla \psi) \right\}. \label{t3.3}
\end{eqnarray}
The assumption that the vorticity is constant near the boundary is useful to tackle the $\curl D_{\chi}^k u$.
Let us stress in particular that 
\begin{eqnarray*}
\curl D_\chi u & =&  D_\chi \curl u + \as \left\{(\nabla  (\chi u))\cdot(\nabla u) \right\} ,
\\ 	 & =&  D \curl u + \as \left\{(\nabla  (\chi u))\cdot(\nabla u) \right\} ,
\\  & =&  \as \left\{(\nabla  (\chi u))\cdot(\nabla u) \right\} .
\end{eqnarray*}
The proof of Theorem  \ref{Dmodifie} then goes as in the  proof of Theorem  \ref{start22}. 
The details are left to the reader.

\section*{APPENDIX A}

 Let $\mathcal{D}$ be an open and bounded subset of $\mathbb{R}^2$. 
Let us denote by $\Gamma (x) := \frac{1}{2\pi} \log   \| x \|$ the fundamental solution  of the Poisson problem in $\mathbb{R}^2$, and for a given function $f \in L^{\infty}  (\mathcal{D} )$ by 
$$ \Psi (x):= (\Gamma * f ) (x) = \int_{\mathcal{D}} \Gamma (x-y) f(y) dy $$
the Newton potential of $f$. 
It is well-known (cf. for instance \cite{GT}, Lemma $4.1$) that $ \Psi \in C^{1}  (\R^2 )$ with 
$$ \nabla  \Psi (x) = \int_{\mathcal{D}} ( \nabla \Gamma) (x-y) f(y) dy .$$
It is well-known (cf. for example \cite{GT}, Exercice $4.1$) that in general $f \in C^0 (\mathcal{D})$ does not imply that $ \Psi  \in C^2 (\mathcal{D})$. However we have the following.
\begin{Lemma}
\label{SchauderDiniPo}
If  $f  \in C_{\mu,\text{loc}}  (\mathcal{D})$  with $\mu$ a Dini modulus of continuity, then 
$ \Psi  \in C^2 (\mathcal{D})$ and $ \Delta  \Psi  = f$ in $\mathcal{D}$.
\end{Lemma}
\begin{proof}
Let  $ \mathcal{D}_0 \supset \mathcal{D}$ be a bounded open set with smooth boundary
$\partial \mathcal{D}_0$, and for $1 \leqslant i,j \leqslant 2$, let 
$$ u_{ij}  (x) :=  \int_{\mathcal{D}_0} ( \partial_{ij} \Gamma ) (x-y) (f(y) - f(x))dy
- f(x)  \int_{\partial \mathcal{D}_0} (\partial_{i}  \Gamma ) (x-y) \nu_j (y) ds(y),
$$
where $\nu$ is the outward normal unit to $\partial \mathcal{D}_0$, and where $f$ is extended by zero outside $\mathcal{D}$. 
The function  $ u_{ij}$  is well-defined for $x \in \mathcal{D}$: the integrands are smooth except when $y$ is in a neighborhood of $x$ in the first integral, say in an open ball $B(x,R) $ such that $\overline{B(x,R)}  \subset \mathcal{D}$. Since $\Gamma$ satisfies the bound 
$|\partial_{ij} \Gamma (x)  | \leq \frac{1}{2\pi \| x\|^2  }$ and 
 since  $f  \in C_{\mu,\text{loc}}  (\mathcal{D})$, the contribution of the ball $B(x,R) $ to the  integral is $\int_0^{R }  \frac{\mu(r)}{ r}  dr < + \infty $, since $\mu$ is a Dini modulus of continuity.

Let $\eta  \in C^{\infty}  (\lbrack 0,\infty) )$ satisfying $\eta (s) = 0$ for $0 \leqslant s \leqslant 1$, $0 \leqslant \eta' (s) \leqslant  2$ for  $1 \leqslant s \leqslant 2$ and $\eta (s) = 1$ for $s \geqslant 2$.
For $i=1,2$, the functions 
$$ v_{i,\eps} (x) :=  \int_{\mathcal{D}} \Gamma_{i,\eps} (x-y) f(y) dy,
\text{ where }  \Gamma_{i,\eps} (x) :=  (\partial_{i}  \Gamma ) (x) \eta ( \| x \| / \eps ), $$
converges uniformly on the compact subsets  of  $\mathcal{D}$ to $\partial_{i} \Psi$ when $\eps >0$ tends to $0$, since for any $x \in \mathcal{D}$ and for $2 \eps  \in  (0,d(x,\partial \mathcal{D} ))$,
$$  v_{i,\eps} (x) -  \partial_{i} \Psi (x)
=  \int_{ B(x,2 \eps)} ( \partial_{i} \Gamma ) (x-y) ( \eta ( \| x-y \| / \eps ) - 1)  f(y)  dy $$
hence
$$ | v_{i,\eps} (x) -  \partial_{i} \Psi (x) | \leqslant 
 \int_{ B(x,2 \eps)}  \frac{1}{2 \pi   \| x-y \| } 2  |f(y)|  dy  
 \leqslant 2 \eps   \| f \|_{L^{\infty}  (\mathcal{D} )} .$$
 %
Now, for any $x  \in \mathcal{D}$,
$$\partial_{j} v_{i,\eps} (x) =   \int_{\mathcal{D}_0}  (\partial_{j}  \Gamma_{i,\eps}) (x-y)  (f(y) - f(x)) dy
+ f(x)  \int_{ \mathcal{D}_0} (\partial_{j}  \Gamma_{i,\eps} ) (x-y)  dy .
$$
Moreover, thanks to Green's identity, we have 
 \begin{eqnarray*}
  \int_{ \mathcal{D}_0} (\partial_{j}  \Gamma_{i,\eps} ) (x-y)  dy = 
-  \int_{ \partial \mathcal{D}_0}   \Gamma_{i,\eps}  (x-y)  \nu_j (y) ds(y) 
\\ = -  \int_{ \partial \mathcal{D}_0} (\partial_{i}  \Gamma ) (x-y)  \nu_j (y) ds(y)  
 \end{eqnarray*}
for $2\eps  \in (0, d(x,\partial \mathcal{D} ))$. 
Therefore
$$\partial_{j} v_{i,\eps} (x) =   \int_{\mathcal{D}_0}  (\partial_{j}  \Gamma_{i,\eps}) (x-y)  (f(y) - f(x)) dy
- f(x)  \int_{\partial \mathcal{D}_0} (\partial_{i}  \Gamma ) (x-y) \nu_j (y) ds(y) .
$$
Then
$$ u_{ij}  (x) - \partial_{j} v_{i,\eps} (x) =  \int_{ B(x,2 \eps)}   \tilde{\Gamma}_{ij,\eps} (x-y)   (f(y) - f(x)) dy ,$$
with
$$\tilde{\Gamma}_{ij,\eps} (x) := \partial_{ij}  \Gamma (x) (1-   \eta ( \| x \| / \eps ))   -   \partial_{i}  \Gamma (x) \eta' ( \| x \| / \eps )  \frac{x_j }{\eps  \| x \| } .$$
Since $| \tilde{\Gamma}_{ij,\eps} (x) | \leqslant    \frac{1}{ \pi } ( \frac{1}{ \| x \|^2  } +  \frac{1}{\eps \| x \|  } )$, we obtain
$$| u_{ij}  (x) - \partial_{j} v_{i,\eps} (x) | \leqslant  6 \int_0^{2 \eps }  \frac{\mu(r)}{ r}  dr $$
which tends to $0$, since $\mu$ is a Dini modulus of continuity.
We therefore have shown that $\partial_{j} v_{i,\eps}$ converges to $ u_{ij}$ when $\eps$ tends to $0$ uniformly on the compact subsets  of  $\mathcal{D}$.
Therefore $ \Psi  \in C^2 (\mathcal{D})$ and $ \partial_{ij} \Psi =  u_{ij} $. 
It is then sufficient to use that  $ \Delta  \Gamma =  \delta_0$ and Green's identity to get  $ \Delta  \Psi  = f$ in $\mathcal{D}$.
\end{proof}

\section*{APPENDIX B}

The goal of this appendix is to provide an explicit proof of Theorem \ref{lin}.
In particular we will show how the motions of isolated point vortices can be considered as  weak solutions of the Euler equations, thanks to an appropriated  weak vorticity formulation of the Euler equations for multiply connected domains. 
Since the trajectories of the point vortices are analytic (up to the first collision) this will provide some examples of very singular solutions of the Euler equations for which the flow restricted to the finite collection of the initial positions of the vortices is analytic, despite the boundary of the domain is not analytic. We  only assume here that  the boundary  $\partial \Omega$  of  the domain $\Omega$ is $C^2$.
The study of the motion of  isolated vortices  goes back to Helmholtz, Kirchoff,  Routh, and to Lin  \cite{lin} in the case of multiply connected domains that will be considered here. 
Let us first recall the existence of  the hydrodynamic Green function.
\begin{Lemma}
\label{green}
There exists an unique function $G:(x,y) \in   \overline{ \Omega} \times  \Omega \mapsto G(x,y) \in \R$ satisfying the following properties:
\begin{enumerate}[(i).]
\item The function 
\begin{eqnarray}
\label{robin}
g (x,y) := G (x,y) - \frac{1}{2 \pi }  \log \|x-y \| 
\end{eqnarray}
 is harmonic with respect to $x$ on $ \Omega $, for any $y \in \Omega $.
\item For  $1\leq l \leq d$, for  $y \in \Omega$, 
the function $G( \cdot ,y)$ is constant when $x$ ranges over  $ C_l$.
\item[(iii).]  The function $G$ vanishes over the outer boundary $C_0$: 
 for  $x \in C_0$, for $y  \in  \Omega$,  $G(x,y) = 0$.
\item[(iv).] For  $1\leq l \leq d$, for  $y \in \Omega$,  the circulation around $C_l$ of $ \nabla^\perp G ( \cdot ,y)$ vanishes: $  \Gamma_l ( \nabla^\perp G ( \cdot ,y)) = 0$.

\end{enumerate}
Moreover $G$ satisfies the reciprocity-symmetry relation: for any $(x,y) \in  { \Omega} \times  \Omega $, 
\begin{eqnarray}
\label{sy}
G (x,y ) = G (y,x) .
\end{eqnarray}
\end{Lemma}
\begin{proof}
We will use again
 for $1\leq i \leq d$ the function $ \phi_i $ in $C^{\infty}(\Omega)$ such that 
$\Delta  \phi_i = 0$ in $\Omega$,  with $\phi_i = \delta_{i,j}$, on $C_j$, for $j = 0, \ldots, d$. 
We introduce the matrix $M := (m_{i,j})_{1\leqslant  i,j \leqslant  d}$ with $m_{i,j} :=  \Gamma_i ( \nabla^\perp \phi_j )$.
Let us also recall that
Green's identity yields for  any smooth  vector field  $f$  from $\Omega$ to $ \R^2$, 
\begin{equation}
\label{charlot33}
\int_\Omega   \phi_i \curl f  +  \Gamma_i (f)  =
 -  \int_\Omega   \nabla^\perp \phi_i \cdot f .
 \end{equation}
This yields in particular $m_{i,j} = - \int_\Omega   \nabla^\perp \phi_i \cdot  \nabla^\perp \phi_j $. Therefore the matrix $M := (m_{i,j})_{1\leqslant  i,j \leqslant  d}$ is symmetric definite negative. Let us denote $p_{i,j}$ the entries of its inverse $M^{-1}$. 
Let us denote by $G_0 (x,y) $ the Green's function associated to the Dirichlet condition.
We then set 
$$G(x,y) := G_0 (x,y) + \sum_{1 \leq i,j  \leq d} p_{i,j}  \phi_i  (x)  \phi_j  (y).$$
The conditions $(i)$,  $(ii)$, $(iii)$ and  the reciprocity-symmetry relation \eqref{sy} are therefore satisfied.
Now \eqref{charlot33} also applies to $f:= \nabla^\perp  G_0 $ and we get  for  $1\leq l \leq d$, for  $y \in \Omega$, 
$$ \Gamma_l ( \nabla^\perp G_0 (\cdot,y )) = - \phi_l  (y) , \text{ and }
 \Gamma_l (\sum_{1 \leq i,j  \leq d} p_{i,j}  \phi_i  (\cdot )  \phi_j  (y)) =
\sum_{1 \leq i,j  \leq d} p_{i,j} m_{l,i}   \phi_j  (y) =  \phi_l  (y) ,
$$
so that
$$  \Gamma_l ( \nabla^\perp G (\cdot,y )) = 0 .$$
The  condition $(iv)$ is therefore satisfied.
Let us now prove the uniqueness: assume that two functions $G_1$ and $G_2$ satisfy the properties  $(i)$,..., $(iv)$ then  for any $y \in \Omega $, Green's identity yields  that
 $ \int_\Omega  \| \nabla (G_1 - G_2 )   \|^2 = 0$, so that $G_1 = G_2$ since they both vanish on $C_0$.
 \end{proof}

We now consider $N$ vortices of respective strength $\alpha_i \in \R^*$, for $1 \leqslant  i \leqslant  N$, located at $N$ distinct points of $\Omega $ and we prescribe some real $\overline{\Gamma}_i$ as respective circulations on the inner  boundaries  $C_i$, for $1 \leqslant  i \leqslant  d$.
We deduce from Lemma \ref{green} that there exists only one corresponding stream function.
\begin{Lemma}
\label{courant}
Let be given 
$$N   \in \N^*,\  \overline{\Gamma} :=  (\overline{\Gamma}_l )_{1\leqslant  l \leqslant  d}  \in \R^d , \    \alpha := (\alpha_l )_{1\leqslant  l \leqslant  N} \in \R^N$$
 and 
 $$\overline{x} := (x_1 ,... ,x_N)   \in   \Omega_N :=  \{ (x_1  ,... ,x_N)  \in \Omega / \quad x_i \neq x_j  \text{ for } 1 \leqslant  i  \neq j \leqslant  N  \}.$$
Then there exists a unique function $\psi :\Omega \mapsto \R $ such that 
\begin{enumerate}[(i).]
\item The function 
\begin{eqnarray*}
\psi  - \sum_{i=1}^N 
 \frac{\alpha_i }{2 \pi }  \log \| \cdot - x_i \| 
\end{eqnarray*}
 is harmonic in $ \Omega $.
\item For  $1\leq l \leq d$, the function $\psi $ is constant when $x$ ranges over  $ C_l$.
\item[(iii).]  The function $\psi $ vanishes over the outer boundary $C_0$: 
 for  $x \in C_0$,  $\psi (x)  = 0$.
\item[(iv).] For  $1\leq l \leq d$,   the circulation around $C_l$ of $ \nabla^\perp \psi  $ is $  {\Gamma}_l ( \nabla^\perp  \psi ) = \overline{\Gamma}_l  $.
\end{enumerate}
Moreover 
\begin{eqnarray}
\label{psi0}
\psi := \sum_{ l =1}^N \alpha_l G ( \cdot , x_l ) + \psi_0 \text{ with }
 \psi_0 := \sum_{ 1 \leq i,j  \leq d } \overline{\Gamma}_i p_{  i,j } \phi_j .
  \end{eqnarray}
\end{Lemma}
The total kinetic energy $ \int_\Omega  \| \nabla \psi \|^2  $ of the flow is infinite (except if all the $\alpha_l$ vanish) so that no information can be derived from its conservation. Nevertheless there exists the following substitute.

%
\begin{Definition}
\label{W}
We  define the  Kirchoff-Routh-Lin function (for  $\overline{x} := (x_1 ,... ,x_N)   \in   \Omega_N$)
 by
\begin{eqnarray*}
 W ( \overline{x} ) :=  \sum_{1 \leqslant l  \leqslant N}  \alpha_l   \psi_0 (x_l)
  +  \frac{1}{2}  \sum_{1 \leqslant l  \leqslant N} \alpha_l^2 r (x_l )
 + \frac{1}{2}  \sum_{1 \leqslant l \neq m \leqslant N} \alpha_l \alpha_m G (x_l , x_m ) 
  ,
 \end{eqnarray*}
where the function $r$ is  the restriction  on its diagonal of the function $g$ appearing in  \eqref{robin}, that is the function $r$ defined on $ \Omega$ by 
$r(x) := g(x,x)$.
 The function $r$ is referred as the hydrodynamic Robin function.
\end{Definition}
Indeed  the  Kirchoff-Routh-Lin function $W $ is a renormalized energy of the system, excluding the free part (that is the one which should take place in the absence of boundaries) of the self-interaction of each vortex. The first term in the definition of $W$ corresponds to the energy created by the interaction with vortices outside $\Omega$
corresponding to the circulations on the $C_l$, the second term correspond to the part of the self-interaction of each vortex induced by the presence of boundaries (by symmetry breaking) and the third one corresponds to the interaction between any distinct pair of vortices. 

\begin{Definition}[Lin \cite{lin}]
\label{def?}
The trajectories $z(t) :=(z_1 (t),...,z_N (t))$  of $N$ point vortices of respective strength $\alpha_i \in \R^*$, for $1 \leqslant  i \leqslant  N$, located at initial time at the  $(x_1  ,... ,x_N)  \in \Omega_N $ is given by the following Hamiltonian ODE
\begin{eqnarray}
\label{motion1}
 \frac{d}{dt} z(t) = F (z(t)),
\\ \label{motion2} z(0) = (x_1 ,... ,x_N) ,
\end{eqnarray}
where $F :z:=(z_1 ,... ,z_N) \in \Omega_N \mapsto F (z) := (F_1 (z),... ,F_N (z))$, with 
 $F_i (z) :=  \frac{1}{ \alpha_i} \nabla^\perp_{x_i} W (z)$. 
  \end{Definition}
 Observing that  the vector field $F$ is  analytic  on  $\Omega_N$ we have the following.
\begin{Lemma}
There exists $T>0$  and a unique solution $z(t) $ in $C^\omega (\lbrack 0,T \rbrack)$  of \eqref{motion1}-\eqref{motion2}.
\end{Lemma}
Let us now introduce an appropriated  weak vorticity formulation of the Euler equations for multiply connected domains. 
\begin{Definition}
\label{WV}
Let be given $\omega_N$ in the space $ \mathcal{M} (\Omega)$ of the Radon measures on $\Omega$.
We say that $ \omega$ in $L^\infty (\R_+ , \mathcal{M} (\Omega))$ is a weak solution of the Euler equations on $ \lbrack 0,T )$ with  $\omega_0$ as initial vorticity and circulations  $\overline{\Gamma}$ if for any test function    
$\varphi \in  C^\infty_c ( \lbrack 0,T )  \times \Omega ,\R )$, 
\begin{eqnarray}
 \label{WF}
 \int_{\Omega }   \varphi(0,x)  d\omega_0 (x) + 
\int_{ \lbrack 0,T \rbrack} \int_{\Omega } L_{\varphi}  (t,x) \, d \omega (t,x)    dt +   
\int_{ \lbrack 0,T \rbrack} \int_{\Omega } \int_{\Omega } H_{\varphi} (t,x,y) \, d\omega (t,x)  d\omega (t,y)   dt  = 0,
\end{eqnarray}
 where $ L_{\varphi }  $ is the function in $C^\infty_c ( \lbrack 0,T )  \times \Omega ,\R ) $, 
 defined by
\begin{align}
 \label{WV3} L_{\varphi } (t,x) :=    \partial_t \varphi (t,x)  +  \bX_0 (x) \cdot  \nabla_x  \varphi (t,x)   ,\end{align}
and  $H_{\varphi}$ is the auxiliary function:
\begin{align}
  \label{WV4} H_{\varphi} (t,x,y) := 
  \left\{
\begin{array}{ll}
  \frac{1}{2} \Big( \nabla_x  \varphi (t,x) \cdot K (x,y) + \nabla_x  \varphi (t,y) \cdot K (y,x)    \Big) ,  &  \mbox{ for  }   x \neq y  \\
   \frac{1}{2}  \nabla_x  \varphi (t,x) \cdot  \nabla_x^\perp  r (x) , &  \mbox{ for  }   x = y ,
\end{array}
\right.
\end{align}
with
\begin{eqnarray}
K (x,y) :=  \nabla_x^\perp G (x,y) \text{ and }
 \bX_0 (x) := \nabla_x^\perp   \psi_0 (x) ,
 \end{eqnarray}
  where the function $ \psi_0 $ is the one in  \eqref{psi0}.
\end{Definition}
The three terms in  \eqref{WF} makes sense: in particular let us observe that the function 
$H_{\varphi}$ is bounded.

Let us first verify that a smooth solution of the Euler equations is also a weak solution in the sense above.
\begin{Lemma}
Let be given $u_0 \in C^\infty ( \lbrack 0,T \rbrack  \times \Omega  ) \cap C ( \lbrack 0,T \rbrack  \times \overline{\Omega }  )  $ satisfying 
$\text{div }  u_0 =0$ in $\Omega$ and $u_0 \cdot \hat{n} = 0$ on  $\partial \Omega $.
Let $u$ be the unique solution in $C^\infty ( \lbrack 0, + \infty )  \times \Omega  )$ of the Euler equations  \eqref{2DincEuler} (cf.  \cite{kato}).
Then for any $T>0$, $\omega := \curl u$ satisfies the  weak vorticity formulation of Definition \ref{WV} with  $\omega_0 := \curl u_0$ an $ \psi_0 := \sum_{ 1 \leq i,j  \leq d } {\Gamma}_i (u_0) p_{  i,j } \phi_j .$
\end{Lemma}
\begin{proof}
 We start with the vorticity formulation of the Euler equations:
\be \label{vorteq}
\partial_t \omega +  \mbox{ div } ( \omega u)= 0 .
\ee
We consider $T>0$ and we multiply by a test function    
$\varphi \in  C^\infty_c ( \lbrack 0,T )  \times \Omega ,\R )$, and integrate by parts over $ \lbrack 0,T \rbrack \times \Omega$ to get
\begin{eqnarray}
 \label{WFpo}
 \int_{\Omega }   \varphi(0,x)  \omega_0 (x) dx+ 
\int_{ \lbrack 0,T \rbrack} \int_{\Omega } \partial_t {\varphi}  (t,x) \,  \omega (t,x) dx   dt +   
\int_{ \lbrack 0,T \rbrack} \int_{\Omega } \nabla_x  {\varphi}  (t,x) \, \omega (t,x) u(t,x) dx dt  = 0 ,
\end{eqnarray}
where $\omega_0$ is the initial value of $\omega$.
Now the velocity can be recovered from the vorticity by using   Lemma \ref{green}.
More precisely we have the following.

\begin{Lemma}
Let be given $\omega \in C^\infty_c ( \Omega  )$ and some real $ \overline{\Gamma}_i$, for $ 1 \leq i  \leq d$. Then there exists a unique $u \in C^\infty( \Omega  ) \cap C ( \overline{\Omega }  )  $ such that 
\be 
\left\{
\begin{array}{ll}
\text{curl }  u =  \omega, &   \mbox{ in }  \Omega,\\
\text{div }  u =0, &   \mbox{ in }  \Omega,\\
u \cdot \hat{n} = 0,  &  \text{ on }  \partial \Omega ,\\
\Gamma_i (u) = \overline{\Gamma}_i , &   \mbox{ for }  i=1,...,d .
\end{array}
\right.
\ee
Moreover  $u =   \bX_0 + K [ \omega] $,
where $ \bX_0$ is as in 
Definition \ref{WV} and 
\begin{eqnarray}
K [ \omega] (x) :=  \int_{\Omega} K (x,y) \omega (y) dy.
\end{eqnarray}
\end{Lemma}
As a consequence we infer  that
\[\int_{\Omega} \varphi(0,x) \omega_0 (x)\, dx + 
\int_{ \lbrack 0,T \rbrack} \int_{\Omega } L_{\varphi}  (t,x) \, d \omega (t,x)    dt + 
\int_{ \lbrack 0,T \rbrack} \int_{\Omega} \nabla \varphi (t,x) \cdot K [\omega] (t,x) 
     \omega(t,x) \,dxdt = 0,
\]
Now by substituting $K [\omega]$ for its integral expression and subsequently symmetrizing the kernel in the nonlinear term above, we get  \eqref{WF} with
\begin{eqnarray}
\label{RRR}
  \frac{1}{2} \Big( \nabla_x  \varphi (t,x) \cdot K (x,y) + \nabla_x  \varphi (t,y) \cdot K (y,x)    \Big) 
  \end{eqnarray}
  instead of $H_{\varphi} (t,x,y)$.
  Since the integrand is in $L^1 (\lbrack 0,T \rbrack \times \Omega \times \Omega )$, modifying \eqref{RRR} for $H_{\varphi}$ does not modify the value of the integral, so that  for any $T>0$, $\omega := \curl u$ satisfies the  weak vorticity formulation of Definition \ref{WV}.
\end{proof}

Let us now start the proof of Theorem  \ref{lin}: we consider    the  trajectories  $z(t) :=(z_1 (t),...,z_N (t))$ on $\lbrack 0,T \rbrack$ of $N$ isolated point vortices  of respective strength $\alpha_l \in \R^*$, for $1 \leqslant  l \leqslant  N$,
given by Lemma \ref{def?}. We denote by $\omega $ the following function  with measure-values 
\begin{eqnarray*}
t \mapsto  \sum_{1 \leqslant l  \leqslant N}  \alpha_l \delta_{z_l (t) } .
\end{eqnarray*}
Let us consider a test function    
$\varphi \in  C^\infty_c ( \lbrack 0,T )  \times \Omega ,\R )$.
Thanks to the chain rule, we have for any $t \in  \lbrack 0,T \rbrack$, for any $1 \leqslant l  \leqslant N$,
\begin{eqnarray*}
\partial_t ( \varphi (t, z_l (t)) ) =  (\partial_t  \varphi ) (t, z_l (t) )  + z'_l (t) \cdot  \nabla_x  \varphi (t, z_l (t) ) .
\end{eqnarray*}
Let us observe that  \eqref{motion1} amounts to the equations
\begin{eqnarray*}
 z'_l (t) =  \bX_0  (z_l (t)) +  \frac{\alpha_l }{2}  \nabla_x^\perp r (z_l (t)) + \sum_{m \neq l} \alpha_m  \nabla_x^\perp G (z_l (t) , z_m (t)) .
\end{eqnarray*}
Therefore
\begin{eqnarray*}
\partial_t ( \varphi (t, z_l (t)) ) =  L_{\varphi}  (t, z_l (t) ) +    \frac{\alpha_l }{2} \nabla_x^\perp r (z_l (t)) \cdot  \nabla_x  \varphi (t, z_l (t) ) + 
 \sum_{m \neq l}    \alpha_m    \nabla_x^\perp G (z_l (t) , z_m (t)) 
\cdot  \nabla_x  \varphi (t, z_l (t) )  .
\end{eqnarray*}
Now integrate on $ \lbrack 0,T \rbrack$, multiply by $ \alpha_l$ and sum over $1 \leqslant l  \leqslant N$ to get 
\begin{eqnarray*}
0 =   \sum_{1 \leq l  \leqslant N}  \alpha_l \varphi (0,x_l ) 
+ \sum_{1  \leqslant  l  \leqslant N}  \alpha_l
 \int_{ \lbrack 0,T \rbrack }  L_{\varphi}  (t, z_l (t) ) dt
\\ + 
 \int_{ \lbrack 0,T \rbrack } \Big(   \frac{1}{2}  \sum_{1  \leqslant  l  \leqslant N}  \alpha_l^2  \nabla_x^\perp r    (z_l (t) ) 
 + \sum_{1  \leqslant  l \neq m  \leqslant N}  \alpha_l  \alpha_m  \nabla_x^\perp G (z_l (t) , z_m (t)) \Big) \cdot
 \nabla_x \varphi (z_l (t) ) dt ,
\end{eqnarray*}
what, after symmetrizing the last sum, amouts to say that
 $\omega$ is  
a weak solution of the Euler equation with $\omega_0 := \sum_{1 \leqslant l  \leqslant N}  \alpha_l \delta_{x_l} $ as initial data and $\overline{\Gamma}_i $ as repective  circulation  around the curve $C_i$, for $1 \leqslant  i \leqslant  d$.

\section*{APPENDIX C}

The goal of this appendix is to provide a proof of the statement below Theorem \ref{Dmodifie} about analyticity of  the flow  of  classical solutions, whose vorticity is constant in a neighborhood of the boundary, which is only assumed to be $C^2$.
To be more general we assume here that the fluid fills a bounded regular domain $\Omega \subset \R^d$, with $d=2$ or $3$. 
We denote,  for $\lambda$ in $\N$ and $r \in (0,1)$, the space ${C^{\lambda,r}_\sigma ( \Omega ) }$ of divergence free vector fields $u$ in $C^{\lambda,r}  ( \Omega  )$  tangent to the boundary. 
\begin{Theorem} 
\label{DmodifieHolder}
Assume that the boundary  $\partial \Omega$ is $C^2$. 
Assume that $u_0$ in  $C^{\lambda+1,r}_\sigma  ( \Omega  )$, where  $\lambda$ in $\N$ and $r \in (0,1)$.
Assume that $\omega_0 := \curl u_0$ is constant outside of a compact $\underline{K} \subset \Omega $.
Then for any compact $K \subset \Omega $ the flow  $ \Phi$  is in the space $C^\omega  (\lbrack 0, T  \rbrack  ,C^{\lambda + 1 ,r} ( K ))$ of real analytic functions from $\lbrack 0, T  \rbrack$  to $C^{\lambda + 1 ,r} ( K )$.
\end{Theorem}

\begin{proof}
Let us introduce a function $ \chi_0$ in $C_c^{\lambda + 1 ,r} ( \Omega )$   such that  $ \chi_0 |_{\underline{K} \cup K }= 1$. There exists $T>0$ and  only one $ \chi$ in $L^{\infty}  (  \lbrack 0, T  \rbrack  ,C_c^{\lambda + 1 ,r} ( \Omega ))$  such that 
\begin{equation*}
D_{\chi} \chi = 0 ,
\quad  \chi |_{ t=0 } =  \chi_0 , \quad  \text{ where } D_{\chi} := \partial_{t} + \chi u.\nabla .
\end{equation*}
As in the proof of Theorem \ref{start22localh}, we have,  for $k$ in $\N^*$, 
$\hat{n} \cdot D_\chi^k u = 0 $ on the boundary $ \partial \Omega$,
and  $ \Gamma_i ( D_\chi^k u) = 0$, for $1\leq i \leq d$.
Moreover we obtain, by iteration, using the identities  \eqref{t3.0}-\eqref{t3.3},  that for  $k \in \N^*$,  in $\Omega$
\begin{eqnarray*}
\div D_{\chi}^k u=\trace\left\{F_{\chi}^k [u]\right\} \text{ where } 
F_{\chi}^k [u] := \sum_{\theta   \in \mathcal{A}_{k}  } d^1_k (\theta  )  \,  f_{\chi} (\theta)  [u], 
\end{eqnarray*}
where
\begin{gather} 
f_{\chi} ( \theta)  [u] : = \nabla \chi D^{\alpha_1} u \cdot \ldots \cdot  \nabla   \chi  D^{\alpha_{s-1} } u  \cdot \nabla D^{\alpha_s} u ,
\end{gather}
and where the $d^1_k (\theta )$ are integers satisfying  $ |d^1_k ( \theta) | \leqslant \frac{k ! }{\alpha ! }$.
As in the proof of Theorem \ref{start22localh} the assumption that the vorticity is constant near the boundary is useful to tackle the $\curl D_{\chi}^k u$. For instance, we have that, in $\Omega$, 
\begin{eqnarray}
D_{\chi} \curl u = D  \curl u .
\end{eqnarray}
Moreover, using the  identities  \eqref{t3.0}-\eqref{t3.3}, we get 
\begin{eqnarray*}
\curl D_{\chi} u = D_{\chi} \curl u +  \as \left\{(\nabla  (\chi u))\cdot(\nabla u)  \right\} .
\end{eqnarray*}
Since we also that 
\begin{eqnarray*}
0 = \curl D u = D \curl u +  \as \left\{(\nabla  u)\cdot(\nabla u)  \right\} ,
\end{eqnarray*}
we infer that 
\begin{eqnarray*}
\curl D_{\chi} u =  \as \left\{(\nabla  ( (\chi - 1 )u) )\cdot( \nabla u)  \right\} .
\end{eqnarray*}
Then proceeding by iteration, and using the  identities  \eqref{t3.0}-\eqref{t3.3} and $ D_{\chi} \chi = 0$, we obtain that   $k \in \N^*$,  in $\Omega$,
\begin{eqnarray*}
\curl D_{\chi}^k u=\trace\left\{G_{\chi}^k [u]\right\} \text{ where } 
G_{\chi}^k [u] := \sum_{ \tilde{\theta}   \in \tilde{\mathcal{A}}_{k}  } d^2_k ( \tilde{\theta}   )  \,  g_{\chi} ( \tilde{\theta} )  [u], 
\end{eqnarray*}
where $\tilde{\mathcal{A}}_{k}$ denotes the set
\begin{equation*}
\tilde{\mathcal{A}}_{k} := \{ (s, \eps ,  \alpha ) /  \ 2 \leqslant s  \leqslant k+1 ,   \ \eps   \in  \{0,1 \} ^{s-1} \text{ with } | \eps |  = 1 , 
 \text{ and }
\alpha := ( \alpha_1,\ldots, \alpha_s )  \in \N^s / \ | \alpha |    = k+1 - s \},
\end{equation*}
for $\tilde{\theta} = (s, \eps ,  \alpha )   \in \tilde{\mathcal{A}}_{k}$, $g_{\chi} ( \tilde{\theta}  )  [u]$ denotes
\begin{gather}
g_{\chi} ( \tilde{\theta}  )  [u] : = \nabla \big( ( \chi - \eps_1 ) D_{\chi}^{\alpha_1} u \big) \cdot \ldots \cdot  \nabla    \big( ( \chi - \eps_{s-1} )  D_{\chi}^{\alpha_{s-1} } u )  \cdot \nabla D_{\chi}^{\alpha_s} u ,
\end{gather}
and where the $d^2_k ( \tilde{\theta}  )$ are integers satisfying  $ |d^2_k (  \tilde{\theta} ) | \leqslant \frac{k ! }{\alpha ! }$.

Then we estimate  recursively the $D_\chi^k u$, for $k$ in $\N^*$, thanks to classical elliptic estimates in H\"older spaces, as in the proof of Theorem $2$ of  \cite{ogfstt}. 

\end{proof}

\section{Acknowlegment }

I warmly thank O. Glass, J. Hulshof and J. Kelliher for useful discussions.

\end{document}